\documentclass[11 pt,reqno]{article}

\usepackage{amsmath}
\usepackage{amstext}
\usepackage{enumerate}
\usepackage{amsfonts}
\usepackage{amscd}
\usepackage{amssymb}
\usepackage{amsthm}
\usepackage{tikz-cd}

\usepackage[top=1in, bottom=1in, left=1.5in, right=1.5in]{geometry}

\usepackage[hidelinks]{hyperref}

\newcommand{\Rep}{\operatorname{Rep}}
\newcommand{\im}{\operatorname{im}}
\newcommand{\id}{\operatorname{id}}
\newcommand{\Hom}{\operatorname{Hom}}

\newcommand{\Ext}{\operatorname{Ext}}
\newcommand{\rk}{\operatorname{rk}}
\newcommand{\Fl}{\operatorname{Fl}}
\newcommand{\pt}{\operatorname{pt}}
\newcommand{\Eq}{\operatorname{Eq}}
\newcommand{\Sq}{\operatorname{Sq}}
\newcommand{\Spec}{\operatorname{Spec}}

\newcommand{\hocolim}{\operatorname{hocolim}}
\newcommand{\coker}{\operatorname{coker}}
\newcommand{\Nil}{\operatorname{Nil}}
\renewcommand{\top}{\operatorname{top}}

\newcommand{\C}{\mathbb{C}}

\newcommand{\F}{\mathbb{F}}
\newcommand{\Z}{\mathbb{Z}}

\newcommand{\K}{\mathcal{K}}
\newcommand{\U}{\mathcal{U}}
\newcommand{\A}{\mathcal{A}}
\renewcommand{\AA}{\mathbb{A}}
\newcommand{\N}{\mathcal{N}}
\newcommand{\I}{\mathcal{I}}
\newcommand{\E}{\mathcal{E}}
\newcommand{\mO}{\mathcal{O}}

\theoremstyle{plain}
  \newtheorem{thm}{Theorem}[section]
  \newtheorem*{thm*}{Theorem}
  \newtheorem{prop}[thm]{Proposition}
  \newtheorem{lem}[thm]{Lemma}
  \newtheorem{cor}[thm]{Corollary}
  \newtheorem*{cor*}{Corollary}
\theoremstyle{definition}
  \newtheorem{rem}[thm]{Remark}

\theoremstyle{remark}
  
  \newtheorem*{defn*}{Definition}
  
  \newtheorem*{conjecture*}{Conjecture}
  
  \newtheorem*{note*}{Notation}
  
  \newtheorem*{question*}{Question}

\title{Lannes's $T$-functor and equivariant Chow rings}
\author{David Hemminger}

\newcommand{\address}{{
  \small

  \textsc{UCLA Mathematics Department, Box 951555, Los Angeles, CA 90095-1555}\par\nopagebreak
  \texttt{dhemminger@math.ucla.edu}
}}

\begin{document}

\maketitle

\begin{abstract}
For $X$ a smooth scheme acted on by a linear algebraic group $G$ and $p$ a prime, the equivariant Chow ring $CH^*_G(X)\otimes \F_p$ is an unstable algebra over the Steenrod algebra.  We compute Lannes's $T$-functor applied to $CH^*_G(X)\otimes \F_p$.  As an application, we compute the localization of $CH^*_G(X)\otimes \F_p$ away from $n$-nilpotent modules over the Steenrod algebra, affirming a conjecture of Totaro as a special case.  The case when $X$ is a point and $n = 1$ generalizes and recovers an algebro-geometric version of Quillen's stratification theorem proved by Yagita and Totaro.
\end{abstract}

A celebrated result in group cohomology is Quillen's stratification theorem \cite[Theorem 6.2]{Quillen71}, which says that for $G$ a compact Lie group and $p$ a prime, the map 
$$H^*(BG,\F_p) \rightarrow \lim_{E \subseteq G}H^*(BE,\F_p)$$
induced by restriction maps is an \emph{F-isomorphism} --- every element of the kernel is nilpotent, and some $p$-power of every element of the codomain lies in the image.  Here the limit is taken over the subgroups $E \subseteq G$ where $E \cong (\Z/p)^k$ for some $k\ge 0$, called \emph{elementary abelian subgroups}, with maps given by conjugations and inclusions.  We will write $\E$ for the category of elementary abelian $p$-groups and $H^*_G$ for $H^*(BG,\F_p)$.

Henn, Lannes, and Schwartz generalized Quillen's theorem in \cite{HLS93} using the theory of unstable modules over the Steenrod algebra, which is summarized in Section \ref{sec:unstable_modules}.  The cohomology ring $H^*(X, \F_p)$ of a space $X$ is naturally an object of $\U^{\top}$, an abelian full subcategory of modules over the mod $p$ Steenrod algebra $\A$ called the category of unstable modules over $\A$.  There is a sequence of localizing subcategories $\U^{\top} \supseteq \Nil_1^{\top} \supseteq \Nil_2^{\top} \supseteq ...$, where $\Nil_n^{\top}$ is the smallest localizing subcategory of $\U^{\top}$ containing all $n$th suspensions of unstable modules and $\bigcap_{n = 1}^\infty \Nil_n^{\top} = 0$.  Theorem 5.4 of \cite{HLS93} computes the localization $L_n^{\top}(H^*_G)$ of $H^*_G$ away from $\Nil_n^{\top}$ for a compact Lie group $G$ in terms of the rings $H^*_E$ and the rings $H^*_{C_G(E)}$ in degrees less than $n$, where $E$ ranges over the elementary abelian subgroups of $G$.  For $n = 1$, the theorem says that $L_1^{\top}(H^*_G) \cong \lim_E H^*_E$, which implies Quillen's theorem.

To what extent do analogs of these results hold in algebraic geometry?  Morel-Voevodsky and Totaro defined the Chow ring $CH^*BG$ of the classifying space of an affine group scheme $G$ over a field $k$ (\cite{MV99}, \cite{Totaro99}).  Let $CH^*_G = CH^*BG \otimes_\Z \F_p$.  More generally, Edidin-Graham defined the equivariant Chow ring of a smooth $G$-scheme X over a field $k$ (\cite{EG98}); let $CH^*_GX$ denote this ring modulo $p$.

Compared to group cohomology, much less is known about Chow rings of classifying spaces. The groups $CH^i_G$ can be infinitely generated $\F_p$-vector spaces, even when $G$ is finite and $k$ is a field of characteristic zero \cite[Corollary 3.2]{Totaro16}, and it is an open question whether $CH^i_G$ must be finite-dimensional when $k$ is algebraically closed.  Furthermore, there is no known algorithm for computing the mod $p$ Chow ring of a finite group.  Nevertheless, some results for group cohomology still hold for Chow rings.  Yagita proved an analog of Quillen's theorem (\cite[Theorem 3.1]{Yagita02}): for $G$ a \emph{finite} group over $\C$, the map
$$CH^*_G \rightarrow \lim_{E \subseteq G}CH^*_E$$
is an F-isomorphism.  Since the cycle map $CH^*_E \rightarrow H^*_E$ is an F-isomorphism for all $E \in \E$, it follows that $CH^*_G \rightarrow H^*_G$ is an F-isomorphism for all finite groups $G$.  Totaro generalized Yagita's theorem to apply for fields $k$ of characteristic not $p$ and containing the $p$th roots of unity \cite{Totaro14}.

Brosnan and Voevodsky constructed Steenrod operations on the Chow ring (\cite{Brosnan03}, \cite{Voevodsky03}), and for any affine group scheme $G$, $CH^*_GX$ can be viewed as an object in the full subcategory of $\U^{\top}$ containing the unstable modules concentrated in even degrees.  Suppose $G$ is a linear algebraic group over a field $k$ of characteristic not $p$ containing the $p$th roots of unity.  In this paper, we compute the localization of $CH^*_GX$ away from $\Nil_n$, the subcategory of $n$-nilpotent unstable modules concentrated in even degrees.

\begin{thm*}[see Theorem \ref{thm:localization-away-from-nil-p}]
There is a map from $CH^*_GX$ to an equalizer
$$\begin{tikzcd}
\Eq \colon \displaystyle\prod_{E}CH^*_E \otimes CH^{<n}_{C_G(E)}X^E \arrow[r,shift left = .75ex,""] \arrow[r,shift right = .75ex, swap, ""]& \displaystyle\prod_{E_1 \rightarrow E_2} CH^*_{E_1} \otimes (CH^*_{E_1} \otimes CH^*_{C_G(E_2)}X^{E_2})^{<n}
\end{tikzcd}$$
realizing localization away from $\Nil_n$.
\end{thm*}

The case when $n = 1$ and $X = \Spec(k)$ generalizes and recovers the results of Yagita and Totaro.  In particular, $G$ may have positive dimension in the following corollary.

\begin{cor*}[see Corollary \ref{cor:quillens-theorem}]
The map
$$CH^*_G \rightarrow \lim_{E\subseteq G} CH^*_E$$
is an F-isomorphism.  Here the limit ranges over all elementary abelian subgroups $E$ of $G$, with maps given by inclusion and conjugation.
\end{cor*}

Another consequence of Theorem \ref{thm:localization-away-from-nil-p} is Proposition \ref{prop:totaro-conjecture}, which positively answers a conjecture of Totaro \cite[Conjecture 12.8]{Totaro14}.  The conjecture describes, when $X = \Spec(k)$, the smallest integer $n$ for which the map of Theorem \ref{thm:localization-away-from-nil-p} is injective in terms of the structure of $CH^*_G$ as a module over $\A$.

A key step to proving Theorem \ref{thm:localization-away-from-nil-p} is to compute $T_V(CH^*_GX)$, where $V \in \E$, and $T_V$ is Lannes's functor $T_V^{\top}$, restricted to the subcategory of $\U^{\top}$ of unstable $\A$-modules concentrated in even degrees.

\begin{thm*}[see Theorem \ref{thm:l_{G,X}}]
Let $G$ act on a smooth scheme $X$.  Then
$$T_V(CH^*_GX) \cong \prod_{[\rho] \in \Rep(V,G)} CH^*_{C_G(\rho)}X^\rho,$$
where $\Rep(V,G)$ is the set of conjugacy classes of group homomorphisms $V \rightarrow G$.
\end{thm*}

For $M \in \U^{\top}$, the integers $d_0^{\top}(M)$ and $d_1^{\top}(M)$ are the smallest nonnegative integers such that the localization maps $\lambda_n^{\top}$ applied to $M$ are injective for $n > d_0^{\top}(M)$ and isomorphisms for $n > d_1^{\top}(M)$.  Henn, Lannes, and Schwartz gave bounds for $d_0^{\top}(H^*_G)$ and $d_1^{\top}(H^*_G)$ in terms of the representation theory of $G$ \cite{HLS93}.  Work of Kuhn and Totaro has resulted in improved bounds for $d_0^{\top}(H^*_G)$ and the analogous invariant $d_0(CH^*_G)$ when $G$ is finite (\cite{Kuhn13}, \cite[Theorem 12.4, Theorem 12.7]{Totaro14}).  We give bounds for $d_0(CH^*_GX)$ and $d_1(CH^*_GX)$ analogous to those of \cite{HLS93} when $G$ is a linear algebraic group over a field $k$ as above.

\begin{thm*}[see Theorem \ref{thm:localization-bounds}]
Suppose that $G$ has a faithful representation of degree $n$.  Then
\begin{enumerate}[a)]
\item $d_0(CH^*_GX) \le n^2 + \dim X - \dim G$ 
\item $d_1(CH^*_GX) \le 2n^2 + \dim X - \dim G$
\end{enumerate}
Suppose further that $G$ is finite.  Then
\begin{enumerate}[a)]
\setcounter{enumi}{2}
\item $d_0(CH^*_GX) \le n(n-1)/2 + \dim X$
\item $d_1(CH^*_GX) \le n(n-1) + \dim X$
\end{enumerate}
\end{thm*}

Together, Theorems \ref{thm:localization-away-from-nil-p} and \ref{thm:localization-bounds} justify the slogan that to understand the Chow ring of (the classifying space of) a linear algebraic group $G$ with $n$-dimensional faithful representation, it suffices to understand the Chow rings of centralizers of its elementary abelian subgroups in degrees less than $2n^2 - \dim G$ (or $n(n-1)$ if $G$ is finite).

The paper is organized as follows.  Section \ref{sec:unstable_modules} gives background information on the theory of unstable modules over the Steenrod algebra, and Section \ref{sec:chow-rings-as-unstable-modules} discusses how Chow rings fit into this theory.  Section \ref{sec:Computation_of_T_VCH^*_GX} is devoted to proving Theorem \ref{thm:l_{G,X}}, which is then applied to study localization of $CH^*_G$ away from $n$-nilpotent modules in Section \ref{sec:localization}.  In Section \ref{sec:d0andd1}, we prove Totaro's conjecture and give bounds for $d_0(CH^*_GX)$ and $d_1(CH^*_GX)$.

\subsection*{Acknowledgements}
I thank Burt Totaro for his advice and encouragement.  I thank Peter Symonds, William Baker, Bar Roytman, and Alex Wertheim for helpful conversations.  I am grateful to the referee for their comments.  This work was supported by National Science Foundation grant DMS-1701237.

\section{Notation}
\label{sec:notation}
Let $p$ be a prime, and let $k$ be a field of characteristic not $p$ containing the $p$th roots of unity.  All schemes are assumed to be separated and of finite type over $k$.  Unless otherwise specified, $G$ will denote a linear algebraic group over $k$, $X$ will denote a smooth $G$-scheme over $k$, and $V$ will denote an elementary abelian $p$-group.

We will often define morphisms between group schemes and equivariant schemes by defining morphisms between their associated functors of points.  A reference for some of the basic constructions for algebraic groups from this point of view is \cite{Milne17}.

We will assume that $X$ is quasi-projective and has an equivariant ample line bundle.  These assumptions are included so that various quotients are guaranteed to exist as schemes by geometric invariant theory; however, they should not be necessary in order for the theorems in this paper to hold.  Though not always a scheme, the quotient of a smooth scheme by a free group action will always exist as a separated algebraic space \cite{KM97}.  It is folklore (see also \cite[Section 6.1]{EG98}, \cite[\href{https://stacks.math.columbia.edu/tag/0EDQ}{Chapter 0EDQ}]{SP}) that enough of the theory of Chow rings of schemes extends in order for the facts used in this paper to hold for separated algebraic spaces.  Note that if a linear algebraic group $G$ acts on a smooth quasi-projective scheme $X$, $X$ has an equivariant ample line bundle if $G$ is finite (\cite[Proposition 3.4.5]{Brion18}) or if $G$ is connected (\cite{Sumihiro74}).

In an abuse of notation, all Chow rings in this paper will be considered modulo $p$ (so that $CH^*X = CH^*X \otimes_\Z \F_p$, where the right hand side uses the usual notation for integral Chow groups).

\section{Unstable modules over the Steenrod algebra}
\label{sec:unstable_modules}
The mod $p$ Steenrod algebra $\A$ is the graded $\F_p$-algebra of natural transformations $H^*(-,\F_p) \Rightarrow H^{* + k}(-,\F_p)$ that commute with the suspension isomorphism $H^*(-,\F_p) \cong H^{*+1}(\Sigma -,\F_p)$.  If $p$ is odd, $\A$ is generated by elements $\beta$ of degree 1 and $P^i$ of degree $2i(p-1)$, $i > 0$, subject to the \emph{Adem relations}
\begin{align*}
 \beta^2 &= 0,\\
 P^aP^b &= \displaystyle\sum_{k = 0}^{\lfloor a/p\rfloor}(-1)^{a+k}\binom{(p-1)(b-k) - 1}{a-pk}P^{a+b-k}P^k, \text{ and } \\
 P^a\beta P^b &= \displaystyle\sum_{k = 0}^{\lfloor a/p \rfloor}(-1)^{a+k}\binom{(p-1)(b-k)}{a-pk} \beta P^{a+b-k}P^k\\
 &\;\;\;\;+ \displaystyle\sum_{k = 0}^{\lfloor(a-1)/p\rfloor}(-1)^{a+k-1}\binom{(p-1)(b-k) - 1}{a-pk-1}P^{a+b-k}\beta P^k
\end{align*}
for all $a < pb$.  If $p = 2$, $\A$ is generated by elements $\Sq^a$ of degree $a$, $a > 0$, subject to the relations
$$\Sq^a\Sq^b = \displaystyle\sum_{k = 0}^{\lfloor a/2\rfloor}\binom{b-k-1}{a-2k}\Sq^{a+b-k}\Sq^k$$
for $a < 2b$.

Steenrod operations satisfy a certain vanishing property: for $p > 2$, if $X$ is a space, $x \in H^*(X,\F_p)$, and $e \in \{0,1\}$, then $\beta^e P^ax = 0$ when $e + 2a > |x|$.  For $p = 2$, $\Sq^ax = 0$ when $a > |x|$.  A graded $\A$-module satisfying the above vanishing property is called an \emph{unstable module over the Steenrod algebra}.  Let $\U^{\top}$ denote the full subcategory of graded $\A$-modules comprising the unstable $\A$-modules.

The cohomology of a space $X$ is an $\F_p$-algebra, and for $x,y \in H^*(X,\F_p)$, the Steenrod operations satisfy the \emph{Cartan formula}
\begin{align*}
  \beta(xy) &= \beta(x)y + (-1)^{|x|}x\beta(y)\\
  P^a(xy) &= \sum_{i = 0}^{a} P^ixP^{a - i}y
\end{align*}
for $p > 2$, and
$$\Sq^a(xy) = \sum_{i = 0}^{a} \Sq^ix\Sq^{a - i}y$$
for $p = 2$.  Finally, we have $P^{|x|/2}x = x^p$ for all $x \in H^{2i}(X,\F_p)$ when $p > 2$ and $\Sq^{|x|}x = x^2$ for all $x \in H^*(X,\F_p)$ when $p = 2$.  An \emph{unstable $\A$-algebra} is an unstable $\A$-module that is also an $\F_p$-algebra satisfying the two above properties.  Let $\K^{\top}$ denote the category of unstable $\A$-algebras.

For an elementary abelian $p$-group $V \cong (\Z/p)^k$, we have 
$$H^*_V \cong \begin{cases}
\F_p[x_1,\ldots,x_k] & \text{if } p = 2 \text{, and}\\
\F_p\langle x_1,\ldots,x_k,y_1,\ldots,y_k\rangle &\text{if } p \text{ is odd.}
\end{cases}$$
where $|x_i| = 1$ and $|y_i| = 2$.  The angle bracket notation means the free graded-commutative algebra on the generators $x_i$ and $y_i$.  The functor $- \otimes H^*_V \colon \U^{\top} \rightarrow \U^{\top}$ has a left adjoint $T_V^{\top}$ that plays a key role in this paper.  

Let $M$ be an unstable $\A$-module.  For $x \in M$, let $\Sq_a x = \Sq^{|x|-a}x$ if $p = 2$.  If $p > 2$ and $|x| = 2k + e$, $e \in \{0,1\}$, then let $P_ax = P^{k-\lfloor a/2\rfloor}x$.  Our interest in $T_V^{\top}$ is motivated by the following proposition. 

\begin{prop}\cite[Proposition 6.1.1]{Schwartz94}\label{prop:n-nilpotents}
Let $n\ge 0$ be an integer.  The following are equivalent.
\begin{enumerate}[a)]
\item $M$ is a colimit of unstable $\A$-modules $M_i$, where each $M_i$ has a finite filtration whose quotients are $n$th suspensions of unstable $\A$-modules.
\item The operations $\Sq_a$ (for $p = 2$) or $P_a$ (for $p > 2$) are locally nilpotent for all $0 \le a < n$.
\item $\Hom_{\U^{\top}}(M,H^*_V \otimes J(k)) = 0$ for all $V$ and for all $0 \le k < n$. 
\item $\Hom_{\U^{\top}}(M,H^*_V \otimes N) = 0$ for all $V$ and for all $N\in \U^{\top}$ concentrated in degrees at most $n - 1$.
\item $(T_V^{\top}M)^i = 0$ for all $V$ and for all $i < n$.
\end{enumerate}
\end{prop}

Here $J(k)$ denotes the $k$th \emph{Brown-Gitler module}, which represents the functor from $\U^{\top}$ to $\F_p$-vector spaces given by $M \mapsto (M^k)^\vee$.  We call an unstable $\A$-module satisfying the conditions in Proposition \ref{prop:n-nilpotents} \emph{$n$-nilpotent}, and $\Nil_n^{\top}\subseteq \U^{\top}$ denotes the full subcategory of $n$-nilpotent unstable $\A$-modules.

The functor $T_V^{\top}$ has surprisingly good properties.  Lannes showed: 

\begin{thm}\cite[Th\'eor\'eme 3.1]{Lannes85}
The functor $T_V^{\top}$ is exact.
\end{thm}

Furthermore, the tensor product of two unstable $\A$-modules over $\F_p$ is again an unstable $\A$-module via the Cartan formula, and Lannes showed that $T_V^{\top}$ commutes with tensor products: 

\begin{lem}\cite[Th\'eor\'eme 3.2.1]{Lannes85}\label{lem:tv-commutes-with-tensor-products}
For $M, N \in \U^{\top}$, the maps $M \rightarrow T_V^{\top} M \otimes H^*_V$ and $N \rightarrow T_V^{\top} N \otimes H^*_V$ adjoint to the identity maps on $T_V^{\top} M$ and $T_V^{\top} N$ induce a map $M \otimes N \rightarrow T_V^{\top} M \otimes H^*_V \otimes T_V^{\top} N \otimes H^*_V$.  Composing with multiplication on $H^*_V$ gives a map adjoint to a natural isomorphism 
 $$T_V^{\top}(M \otimes N) \xrightarrow{\cong} T_V^{\top} M \otimes T_V^{\top} N$$
\end{lem}

Lemma \ref{lem:tv-commutes-with-tensor-products} allows for describing the left adjoint to $- \otimes H^*_V$ at the level of unstable algebras.

\begin{lem}\cite[Theorem 0.1]{Lannes92}
For $K \in \K^{\top}$, $T_V^{\top}K$ is naturally an unstable $\A$-algebra via the map
$$T_V^{\top}(K) \otimes T_V^{\top}(K) \cong T_V^{\top}(K \otimes K) \xrightarrow{T_V^{\top}(\mu)} T_V^{\top}(K),$$
where $\mu$ is the multiplication map $K \otimes K \rightarrow K$.  With this $\F_p$-algebra structure, $T_V^{\top}$ is left adjoint to the functor $- \otimes H^*_V \colon \K^{\top} \rightarrow \K^{\top}$.
\end{lem}

\section{Chow rings as unstable modules over the Steenrod algebra}\label{sec:chow-rings-as-unstable-modules}
For this section, let $G$ be an affine group scheme of finite type over $k$.  The Chow ring of the classifying space of $G$ was defined by Morel-Voevodsky \cite{MV99} and Totaro (\cite{Totaro99}, \cite{Totaro14}), and Edidin-Graham \cite{EG98} extended the definition to equivariant Chow rings.

We say that $G$ acts freely on a scheme $Y$ if the map $G \times Y \rightarrow Y \times Y$, $(g,y) \mapsto (y,gy)$ is an isomorphism to a closed subscheme (this definition differs from others in the literature; cf.\ \cite[\href{https://stacks.math.columbia.edu/tag/07S1}{Tag 07S1}]{SP}).  For each $i$, fix a faithful representation $V_i$ of $G$ such that the open subset $U_i \subseteq V_i$ on which $G$ acts freely has complement with codimension more than $i$.  If $X$ is a smooth scheme with a $G$-action, the equivariant Chow ring of $X$ is defined by $CH^i_GX := CH^i(U_i \times_G X)$ (here the quotient $U_i \times_G X$ exists as a scheme by the assumptions of Section \ref{sec:notation}).  The definition does not depend on the choices of representations and $G$ and enjoys all of the usual functorial properties of Chow rings \cite[Sections 2.2 and 2.3]{EG98}.

\begin{rem}
We think of the $U_i$ and $U_i/G$ as finite-dimensional approximations to algebro-geometric analogs of the spaces $EG$ and $BG$ in topology.  Indeed, we can pick the $V_i$ so that we have inclusions $V_i \rightarrow V_{i+1}$.  Then the \emph{geometric classifying space} $B_{gm}G$ of $G$ and its universal principal $G$-bundle $E_{gm}G$ can be defined in the $\AA^1$-homotopy category $\mathcal{H}(k)$ as $\hocolim_{m}U_m/G$ and $\hocolim_{m}U_m$ \cite[Section 4.2]{MV99}.
\end{rem}  

We will frequently abuse notation and write $EG$, $BG$, and $EG \times_G X$ instead of the schemes $U_i$, $U_i/G$, and $U_i \times_G X$ approximating them.

Let $\U$ denote the full subcategory of $\U^{\top}$ containing the unstable modules concentrated in even degrees.  We will regrade the objects of $\U$ as follows.  For $M \in \U$, if $x \in M$ has degree $2i$ as an element of an unstable module in $\U^{\top}$, it will have degree $i$ as an element of an object of $\U$.

Brosnan constructed Steenrod operations on the mod $p$ Chow ring of a smooth algebraic space over $k$ \cite{Brosnan03}.  Independently, Voevodsky constructed Steenrod operations for motivic cohomology when $k$ is perfect \cite{Voevodsky03}, and Hoyois-Kelly-{\O}stv{\ae}r extended his definition to all fields \cite{HKO17}.  For $X$ a smooth scheme, we have operations
$$P^a \colon CH^iX \rightarrow CH^{i + a(p-1)}X$$
when $p$ is odd, and 
$$\Sq^{2a} \colon CH^iX \rightarrow CH^{i + a}X$$
when $p = 2$ (recall that all Chow groups in this paper are considered modulo $p$).  For Steenrod operations on Chow groups, the operations $\Sq^{2a}$ are sometimes also denoted $P^a$.

The motivic Adem relations --- which were given in \cite{Voevodsky03} with some mistakes that were corrected in \cite[Theorems 4.5.1 and 4.5.2]{Riou12} --- differ from the topological Adem relations when $p = 2$, but the relations agree after restricting to the subring of operations of even topological degree.  In particular, the motivic Adem relations agree with the topological ones after restricting to the Chow ring.  The operations $P^a$ satisfy the Cartan formula, and for $x \in CH^iX$,
$$P^ax = \begin{cases}0 & a > i \\ x^p & a = i\end{cases}$$
Thus $CH^*X$ is an object of $\U$.

For $V\cong (\Z/p)^k$ an elementary abelian $p$-group, we have $CH^*_V \cong \F_p[y_1,\ldots,y_k]$, where $|y_i| = 1$.  The functor $- \otimes  CH^*_V \colon \U \rightarrow \U$ has a left adjoint $T_V$ which is naturally isomorphic to the restriction of $T_V^{\top}$ to $\U$.

\begin{lem}[{\cite[Section 2.2.6]{Lannes92}}]\label{lem:compare-tv-and-tv-prime}
For $M \in \U$, let $\varphi_M \colon T_V^{\top}M \rightarrow T_V M$ be the map adjoint to the composition $M \rightarrow T_VM \otimes CH^*_V \rightarrow T_VM \otimes H^*_V$, where the first map is adjoint to the identity and the second is induced by inclusion.  Then $\varphi_M$ is an isomorphism.
\end{lem}

Since it's isomorphic to the restriction of $T_V^{\top}$, the functor $T_V$ is also exact and commutes with tensor products.

\begin{lem}[{\cite[Section 2.2.6]{Lannes92}}]
For $M, N \in \U$, the maps $M \rightarrow T_V M \otimes CH^*_V$ and $N \rightarrow T_V N \otimes CH^*_V$ adjoint to the identity and the multiplication map $CH^*_V \otimes CH^*_V \rightarrow CH^*_V$ induce a map 
$$M \otimes N \rightarrow (T_V M \otimes CH^*_V) \otimes (T_V N \otimes CH^*_V) \rightarrow T_V M \otimes T_V N \otimes CH^*_V$$
that is adjoint to a natural isomorphism 
$$T_V(M \otimes N) \xrightarrow{\cong} T_V M \otimes T_V N$$
\end{lem}

Note that the discussion in \cite[Section 2.2.6]{Lannes92} refers to the case $p > 2$, but when $p = 2$, the categories $\U$ and $\U^{\top}$ are equivalent: if $\A^{\text{ev}} \subseteq \A$ is the subalgebra generated by the elements $\Sq^{2a}$, $a\ge 0$, an equivalence $\U^{\top} \to \U$ is given by restricting to $\A^{\text{ev}}$ and changing the grading, using that $\A^{\text{ev}} \cong \A$ after regrading.  Furthermore, the $\A$-modules $H^*_V$ and $CH^*_V$ are isomorphic after changing the grading in the $p = 2$ case.

We will need the following lemmas later.

\begin{lem}\label{lem:tensor-product-and-isomorphisms}
Suppose $M, N, K, L \in \U^{\top}$ and $f \colon T_V M \rightarrow K$ and $g \colon T_V N \rightarrow L$ are isomorphisms adjoint to maps $\tilde{f} \colon M \rightarrow K \otimes CH^*_V$ and $\tilde{g} \colon N \rightarrow L \otimes CH^*_V$.  Then the map 
$$M \otimes N \rightarrow (K \otimes CH^*_V) \otimes (L \otimes CH^*_V) \rightarrow K \otimes L \otimes CH^*_V$$
induced by $\tilde{f}$, $\tilde{g}$, and the multiplication map $CH^*_V \otimes CH^*_V \rightarrow CH^*_V$ is adjoint to an isomorphism 
  $$T_V(M \otimes N) \xrightarrow{\cong} K \otimes L$$
\end{lem}

\begin{proof}
We have a commutative diagram
$$\begin{tikzcd}[column sep=tiny]
M \otimes N \arrow[r,""] \arrow[d,""] & T_VM \otimes CH^*_V \otimes T_VN \otimes CH^*_V \arrow[r,""] \arrow[dl,"\cong"] & T_VM \otimes T_VN \otimes CH^*_V \arrow[d,"\cong"]\\
K \otimes CH^*_V \otimes L \otimes CH^*_V \arrow[rr,""] & & K \otimes L \otimes CH^*_V
\end{tikzcd}$$
This gives a commutative square
$$\begin{tikzcd}
M \otimes N \arrow[r,""] \arrow[d,""] & T_VM \otimes T_VN \otimes CH^*_V \arrow[d,""]\\
K \otimes L \otimes CH^*_V \arrow[r,"\id"] & K \otimes L \otimes CH^*_V 
\end{tikzcd},$$
so by adjunction the following square commutes.
$$\begin{tikzcd}
T_V(M \otimes N) \arrow[r,"\cong"] \arrow[d,""] & T_VM \otimes T_VN \arrow[d,"\cong"]\\
T_V(K \otimes L \otimes CH^*_V) \arrow[r,"\epsilon"] & K \otimes L
\end{tikzcd}$$
The desired adjunction follows.
\end{proof}

\begin{lem}\label{lem:tv-for-bounded-above}
If $M \in \U$ is bounded above, the map $T_VM \rightarrow M$ adjoint to the inclusion $M \rightarrow M \otimes CH^*_V$ is an isomorphism.
\end{lem}

\begin{proof}
This follows from Lemma \ref{lem:compare-tv-and-tv-prime} and \cite[Proposition 2.2.4]{Lannes92}, which is the analogous fact for a bounded-above $M \in \U^{\top}$.
\end{proof}

\section{Computation of \texorpdfstring{$T_VCH^*_GX$}{TVCH*GX}}
\label{sec:Computation_of_T_VCH^*_GX}

The purpose of this section is to prove Theorem \ref{thm:l_{G,X}}, which computes $T_VCH^*_GX$.  

Given $\rho \in \Hom(V,G)$, we write $C_G(\rho)$ for the centralizer of the image of $\rho$ in $G$.  The group homomorphism $V \times C_G(\rho) \rightarrow G$, $(v,h) \mapsto \rho(v)h$ induces a map 
$$f_{G,X} \colon BV \times EC_G(\rho) \times_{C_G(\rho)} X^\rho \rightarrow EG \times_G X.$$
This induces a map $CH^*_G(X) \rightarrow CH^*(BV \times EC_G(\rho) \times_{C_G(\rho)} X^\rho)$.  By \cite[Corollary 9.10]{Totaro16}, $BV$ satisfies the Chow K{\"u}nneth property, meaning that the product map $CH^*BV \otimes CH^*X \rightarrow CH^*(BV \times X)$ is an isomorphism for all schemes $X$.  It follows that $CH^*(BV \times EC_G(\rho) \times_{C_G(\rho)} X^\rho) \cong CH^*_V \otimes CH^*_{C_G(\rho)}(X^\rho)$, giving a map $CH^*_GX \rightarrow CH^*_V \otimes CH^*_{C_G(\rho)}(X^\rho)$.  This induces a map
$$\tilde{\ell}_{G,X} \colon CH^*_GX \rightarrow \prod_{[\rho] \in \Rep(V,G)} CH^*_V \otimes CH^*_{C_G(\rho)}(X^\rho),$$
which is adjoint to a map
$$\ell_{G,X} \colon T_VCH^*_G(X) \rightarrow \prod_{[\rho] \in \Rep(V,G)} CH^*_{C_G(\rho)}X^\rho,$$
where $\Rep(V,G)$ is the set of conjugacy classes of group homomorphisms $V \rightarrow G$.  Let $F_1(X)$ and $F_2(X)$ be the domain and codomain of $\ell_{G,X}$, so that $\ell_{G,X}$ defines a natural transformation of two functors $F_1$ and $F_2$ from smooth schemes to $\U$.

\begin{thm}\label{thm:l_{G,X}}
The map $\ell_{G,X}$ is an isomorphism.  
\end{thm}

Let $S \in \E$.  We will first prove Theorem \ref{thm:l_{G,X}} for the map $\ell_{S,\pt}$, then for $\ell_{S,X}$, then for $\ell_{G,X \times GL(n)/S}$, and then finally in full generality.

\begin{lem}\label{lem:l_{S,*}}
The map $\ell_{S,\pt} \colon T_VCH^*_S \rightarrow \prod_{\rho \colon V \rightarrow S} CH^*_S$ is an isomorphism.
\end{lem}

\begin{proof}
The forgetful functor $\K^{\top} \rightarrow \U^{\top}$ has a left adjoint $U$, which sends an unstable $\A$-module $M$ to the symmetric algebra on $M$ quotiented by the ideal generated by elements of the form $\Sq^{|x|}x - x^2$ (for $p = 2$) or $P^{|x|/2}x - x^p$ (for $p > 2$).  These functors restrict to an adjunction on $\K$ and $\U$, where $\K$ is the category of unstable $\A$-modules concentrated in even degrees.

The inclusions $\U \rightarrow \U^{\top}$ and $\K \rightarrow \K^{\top}$ have a left adjoints $\bar{\mO}$ and $\bar{\mO}_K$, given by quotienting out by the submodule or ideal generated by elements in odd degrees.  We then compute that if $F(2)$ is the free unstable $\A$-module on a generator in degree two, then $U\bar{\mO}F(2) \cong CH^*_{\Z/p}$.  For $A \in \K$, by a series of adjunction isomorphisms together with the isomorphism $\Hom_{\U^{\top}}(F(2),A \otimes CH^*_{\Z/p}) \cong \Hom_{\U^{\top}}(F(0) \oplus F(2), A)$ we have that
$$\Hom_{\K}(T_{\Z/p}(CH^*_{\Z/p}), A) \cong \Hom_{\K}(U\bar{\mO}F(0) \otimes U\bar{\mO}F(2), A).$$
Since $F(0) \cong \F_p\{x\}$, we have $U\bar{\mO}F(0) \cong \F_p[x]/(x - x^p) \cong \F_p^p$.  It follows that $T_{\Z/p}(CH^*_{\Z/p}) \cong \prod_{\rho \colon \Z/p \rightarrow \Z/p}CH^*_{\Z/p}$.  Since $CH^*_V \cong (CH^*_{\Z/p})^{\otimes \rk V}$, 
$$T_V(CH^*_{\Z/p}) \cong (T_{\Z/p})^{\circ \rk V}(CH^*_{\Z/p}) \cong \prod_{\rho \colon V \rightarrow \Z/p}CH^*_{\Z/p}.$$  Then $T_V(CH^*_W) \cong \prod_{\rho \colon V \rightarrow W}CH^*_W$ by the fact that $T_V$ commutes with tensor products.

It remains to check that the map $\ell_{W,\pt}$ is an isomorphism. Since the domain and codomain of $\ell_{W,\pt}$ are finitely generated $\F_p$-vector spaces in each degree, it suffices to show that $\ell_{W,\pt}$ is injective.  This follows from the commutative diagram
$$\begin{tikzcd}
T_VCH^*_W \arrow[d,"\ell_{W,\pt}"] & \arrow[l,"\cong"] T_V^{\top}CH^*_W \arrow[r,""] \arrow[d,"\varphi"] & T_V^{\top}H^*_W \arrow[d,"\cong"]\\
\displaystyle\prod_{\rho}CH^*_W & \arrow[l,"\id"] \displaystyle\prod_{\rho}CH^*_W \arrow[r,""] & \displaystyle\prod_{\rho} H^*_W
\end{tikzcd}.$$
Here the group homomorphism $V \times W \rightarrow W$, $(v,w) \mapsto \rho(v)w$ and inclusion induce the composition $CH^*_W \rightarrow CH^*_W \otimes CH^*_V \rightarrow CH^*_W \otimes H^*_V$, which is adjoint to $\varphi$.  The rightmost vertical map is given by the adjoints of the maps $H^*_W \rightarrow H^*_W \otimes H^*_V$ induced by $V \times W \rightarrow W$, $(v,w) \mapsto \rho(v)w$; it is an isomorphism by \cite[Theorem 18]{Henn01}.
\end{proof}

We will need the next lemma in this section and again in Section \ref{sec:d0andd1}.  The exact sequences of parts (a) and (b) may appear more familiar if the reader imagines, in the notation of the lemma, $e_{d,W}CH^*_SX_{d,W}$ to be the equivariant Chow ring of the Thom space of the normal bundle $N_{d,W}$.  With the appropriate equivariant model structure for $\AA^1$-homotopy theory, it should be possible to make this analogy literal, but we will not need this.  One technical issue in doing so would be that the fixed points functor is not homotopy invariant.

\begin{lem}\label{lem:exact-sequences}
Let $X$ be a smooth scheme acted on by an elementary abelian group $S$.  Let $X_i \subseteq X$ be the open subscheme of points with isotropy group of rank at most $i$, $X^{(W)} \subseteq X$ be the subscheme of points with isotropy group $W \subseteq S$, and $X_{d,W}$ be the union of the connected components of $X^{(W)}$ of codimension $d$ in $X$.  Let $N_{d,W}$ be the normal bundle of $X_{d,W}$ in $X_i$, and denote its Euler class $\chi(N_{d,W})$ by $e_{d,W}$.  Then there are short exact sequences of unstable $\A$-modules
\begin{enumerate}[a)]
\item\label{eq:first-exact-sequence}
$$0 \rightarrow \bigoplus_{d,W} e_{d,W}CH^*_SX_{d,W} \rightarrow CH^*_SX_i \rightarrow CH^*_SX_{i-1} \rightarrow 0$$
and
\item\label{eq:second-exact-sequence}
$$0 \rightarrow e_{d,W}CH^*_SX_{d,W} \rightarrow CH^*_SN_{d,W} \rightarrow CH^*_S(N_{d,W} - X_{d,W}) \rightarrow 0.$$
\end{enumerate}
Moreover, there are maps $h_{d,W}$ defining maps of short exact sequences 
\begin{enumerate}[a)]
\setcounter{enumi}{2}
\item\label{eq:first-exact-sequence-map}
$$\begin{tikzcd}
  \displaystyle\bigoplus_{d,W} T_V(e_{d,W}CH^*_SX_{d,W})  \arrow[d,"\oplus h_{d,W}"] \arrow[hookrightarrow]{r} &  T_V(CH_S^*(X_{i})) \arrow[twoheadrightarrow]{r} \arrow[d,"\ell_{S,X_{i}}"] & T_V(CH_S^*X_{i-1}) \arrow[d,"\ell_{S,X_{i-1}}"] \\
  \displaystyle\bigoplus_{d,W} \displaystyle\prod_{\rho} e_{d,W,\rho}CH^*_SX_{d,W}^\rho \arrow[hookrightarrow]{r} & \displaystyle\prod_{\rho} CH_S^*X_{i}^{\rho} \arrow[twoheadrightarrow]{r} & \displaystyle\prod_{\rho} CH_S^*X_{i-1}^{\rho} 
\end{tikzcd}$$
and
\item\label{eq:second-exact-sequence-map}
$$\begin{tikzcd}
T_V(e_{d,W}CH^*_SX_{d,W})\arrow[hookrightarrow]{r}\arrow[d,"h_{d,W}"]&  T_V(CH^*_SN_{d,W})\arrow[twoheadrightarrow]{r}\arrow[d,"\ell_{S,N_{d,W}}"]&  T_V(CH^*_S(N_{d,W} - X_{d,W}))\arrow[d,"\ell_{S,N_{d,W} - X_{d,W}}"]\\
\displaystyle\prod_{\rho} e_{d,W,\rho}CH^*_SX_{d,W}^\rho \arrow[hookrightarrow]{r}& \displaystyle\prod_{\rho} CH^*_SN_{d,W}^\rho \arrow[twoheadrightarrow]{r}& \displaystyle\prod_{\rho} CH^*_S(N_{d,W} - X_{d,W})^\rho
\end{tikzcd}$$
\end{enumerate}
Here $d$ ranges from 0 to $\dim X$, $W$ ranges over all subgroups of $S$ isomorphic to $(\Z/p)^i$, and $\rho$ ranges over all homomorphisms $V \rightarrow S$.
\end{lem}

\begin{proof}
\emph{Proof of (\ref{eq:first-exact-sequence}) and (\ref{eq:first-exact-sequence-map}).} By \cite[Theorem 6.7]{Totaro14} (which is an algebro-geometric analog of \cite[Theorem 1]{Duflot83}), there is an exact sequence
$$ 0 \rightarrow \bigoplus_{d,W} CH_{S}^{* - d}X_{d,W} \xrightarrow{\oplus(i_{d,W})_*} CH_S^*(X_{i}) \rightarrow CH_S^*(X_{i-1}) \rightarrow 0$$
where $i_{d,W}$ is the inclusion $X_{d,W} \rightarrow X_{i}$ and the map $CH^*_S(X_{i}) \rightarrow CH^*_S(X_{i-1})$ is the pullback of the inclusion $X_{i-1} \rightarrow X_{i}$.  Note that pushforward maps generally do not commute with (cohomological) Steenrod operations, so $\oplus_{d,W}(i_{d,W})_*$ is not a map of unstable $\A$-modules.

The map $i_{d,W}^* \circ (i_{d,W})_* \colon CH_S^{*-d}X_{d,W} \rightarrow CH_S^*X_{d,W}$ is given by multiplication by $e_{d,W}$ \cite{Fulton98}.  The proof of \cite[Theorem 6.7]{Totaro14} shows that in this case $e_{d,W}$ is not a zero-divisor in $CH_S^*X_{d,W}$, so that $i_{d,W}^* \circ (i_{d,W})_*$ is an injection.  Thus $\oplus i_{d,W}^*$ defines an isomorphism of $\A$-modules
$$\ker(CH_S^*(X_{i}) \rightarrow CH_S^*(X_{i-1})) \xrightarrow{\oplus i_{d,W}^*} \bigoplus_{d,W}e_{d,W}CH_S^*X_{d,W}$$
Thus we have an exact sequence 
$$0\rightarrow \displaystyle\bigoplus_{d,W} e_{d,W}CH^*_SX_{d,W} \xrightarrow{(\oplus i_{d,W}^*)^{-1}}  CH_S^*(X_{i}) \rightarrow CH_S^*(X_{i-1}) \rightarrow 0,$$
which is the exact sequence (\ref{eq:first-exact-sequence}).  Applying $T_V$ gives the first row of (\ref{eq:first-exact-sequence-map}).  Considering The spaces $X^\rho$ instead of $X$ gives an exact sequence
$$0 \rightarrow \bigoplus_{d,W} \prod_{\rho} e_{d,W,\rho}CH^*_SX_{d,W}^\rho \xrightarrow{\oplus_{d,W,\rho}(i^*_{d,w,\rho})^{-1}} \prod_{\rho} CH_S^*X_{i}^{\rho} \rightarrow  \prod_{\rho} CH_S^*X_{i-1}^{\rho} \rightarrow 0,$$
which is the second row of (\ref{eq:first-exact-sequence-map}).  Here $i_{d,W,\rho}$ denotes the inclusion $X_{d,W} \rightarrow (X_{i})^{\rho}$.

Let $N_{d,W} = \bigoplus_{\psi \in \hom(\im \rho,G_m)} E_\psi$ be the isotypic decomposition of $N_{d,W}$.  Note that $E_0$, the subbundle of $N_{d,W}$ on which $\im \rho$ acts trivially, is $N_{d,W,\rho}$.  Thus $e_{d,W} = e_{d,W,\rho} \cdot \chi (\bigoplus_{\psi \neq 0} E_\psi)$.  Let $W_{\rho} \subseteq S$ be such that $S = \im \rho \times W_\rho$.  Then 
$$CH_S^* X_{d,W}^\rho \cong CH^*_{W_\rho} X_{d,W}^\rho \otimes CH^*_{\im \rho},$$
and since $\im \rho$ acts trivially on $E_0$, $e_{d,W,\rho}$ lies in the subring $CH^*_{W_\rho} X_{d,W}^\rho \otimes 1$ of $CH_S^* X_{d,W}^\rho$.  We have a commutative square
$$\begin{tikzcd}
CH_S^* X_{d,W}^\rho \arrow[r,""] \arrow[d,"\cong"] & CH_S^* X_{d,W}^\rho \otimes CH^*_{V} \arrow[d,"\cong"]\\
CH^*_{W_\rho} X_{d,W}^\rho \otimes CH^*_{\im \rho} \arrow[r,""] & CH^*_{W_\rho} X_{d,W}^\rho \otimes CH^*_{\im\rho} \otimes CH^*_{V} 
\end{tikzcd}$$
where the top map is induced by $V \times S \rightarrow S$, $(v,s) \mapsto \rho(v)s$ and the bottom map is induced by the comultiplication map $CH^*_{\im \rho} \rightarrow CH^*_{\im \rho} \otimes CH^*_{\im \rho}$ and the inclusion $CH^*_{\im\rho} \subseteq CH^*_{V}$.  It follows that 
$$\tilde{\ell}_{S,X_{d,W}}(e_{d,W}) \in \bigoplus_{\rho} e_{d,W,\rho} CH^*_S X_{d,W}^\rho,$$ 
so $\tilde{\ell}_{S,X_{d,W}}$ induces a map $\tilde{h}_{d,W}$ in the commutative square below.
$$\begin{tikzcd}
e_{d,W}CH^*_S X_{d,W} \arrow[r,""] \arrow[d,dashed,"\tilde{h}_{d,W}"] & CH^*_S X_{d,W} \arrow[d,"\tilde{\ell}_{S,X_{d,W}}"]\\
\displaystyle\bigoplus_{d,W}\displaystyle\prod_{\rho} e_{d,W,\rho}CH^*_S X_{d,W}^\rho \otimes CH^*_{V} \arrow[r,""] & \displaystyle\prod_{\rho} CH^*_S X_{d,W}^\rho \otimes CH^*_{V} 
\end{tikzcd}$$
Thus we have a commutative square
$$\begin{tikzcd}[column sep=huge]
e_{d,W}CH^*_S X_{d,W} \arrow[r,"(i^*_{d,W})^{-1}"] \arrow[d,"\tilde{h}_{d,W}"] & CH^*_S(X_{i})  \arrow[d,"\tilde{\ell}_{S,X_{i}}"]\\
\displaystyle\bigoplus_{d,W}\displaystyle\prod_{\rho} e_{d,W,\rho}CH^*_S X_{d,W}^\rho \otimes CH^*_{V} \arrow[r,"(i^*_{d,W,\rho})^{-1} \otimes \id"]& \displaystyle\prod_{\rho} CH^*_S(X_{i})^\rho \otimes CH^*_{V} 
\end{tikzcd}$$
Let $h_{d,W}$ be the map adjoint to $\tilde{h}_{d,W}$.  It follows that the first square in (\ref{eq:first-exact-sequence-map}) commutes, so (\ref{eq:first-exact-sequence-map}) is a map of short exact sequences.

\emph{Proof of (\ref{eq:second-exact-sequence}) and (\ref{eq:second-exact-sequence-map})}
The basic exact sequence for Chow groups gives an exact sequence
$$CH^*_SX_{d,W} \xrightarrow{(\alpha_{d,W})_*} CH^*_SN_{d,W} \xrightarrow{(\beta_{d,W})^*} CH^*_S(N_{d,W} - X_{d,W}) \rightarrow 0$$
where $\alpha_{d,W}$ and $\beta_{d,W}$ are the inclusions $X_{d,W} \rightarrow N_{d,W}$ and $N_{d,W} - X_{d,W} \rightarrow N_{d,W}$.  We again have that $(\alpha_{d,W})^* \circ (\alpha_{d,W})_* \colon CH^*_SX_{d,W} \rightarrow CH^*_SX_{d,W}$ is given by multiplication by $e_{d,W}$.  Since $e_{d,W}$ is not a zero divisor, $(\alpha_{d,W})^*$ defines an isomorphism of $\A$-modules $\ker(CH^*_SN_{d,W} \rightarrow CH^*_S(N_{d,W} - X_{d,W})) \rightarrow e_{d,W}CH^*_SX_{d,W}$.  This gives (\ref{eq:second-exact-sequence}) and exactness of the first row of (\ref{eq:second-exact-sequence-map}), where the first map is given by $(\alpha_{d,W}^*)^{-1}$.
Exactness of the second row of (\ref{eq:second-exact-sequence-map}) follows from a similar argument, with the first map given by $(\alpha_{d,W,\rho}^*)^{-1}$, where the inclusion $\alpha_{d,W,\rho} \colon X_{d,W}^\rho \rightarrow N_{d,W}^\rho$ induces an isomorphism $\ker(CH^*_SN_{d,W}^\rho \rightarrow CH^*_S(N_{d,W} - X_{d,W})^\rho) \rightarrow e_{d,W}CH^*_SX_{d,W}^\rho$.

By the definition of $\tilde{h}_{d,W}$, we have a commutative square
$$\begin{tikzcd}[column sep=huge]
e_{d,W}CH^*_S X_{d,W} \arrow[r,"(\alpha_{d,W}^*)^{-1}"] \arrow[d,"\tilde{h}_{d,W}"] & CH^*_S(N_{d,W})  \arrow[d,"\tilde{\ell}_{S,N_{d,W}}"]\\
\displaystyle\bigoplus_{d,W}\displaystyle\prod_{\rho} e_{d,W,\rho}CH^*_S X_{d,W}^\rho \otimes CH^*_{V} \arrow[r,"(\alpha_{d,W,\rho}^*)^{-1} \otimes \id"]& \displaystyle\prod_{\rho} CH^*_SN_{d,W}^\rho \otimes CH^*_{V} 
\end{tikzcd}$$
It follows that (\ref{eq:second-exact-sequence-map}) is a map of short exact sequences.
\end{proof}

\begin{lem}\label{lem:l_{S,X}}
The map $\ell_{S,X} \colon T_VCH^*_SX \rightarrow \prod_{\rho \colon V \rightarrow S} CH^*_{S}X^\rho$ is an isomorphism.
\end{lem}

\begin{proof}
Let $r(X) = \max_T(\rk T)$, where $T$ ranges over all isotropy subgroups of points on $X$.  If $r(X) = 0$, then $S$ acts freely on $X$, so $CH^*_SX \cong CH^*X/S$, and the claim follows from Lemma \ref{lem:tv-for-bounded-above}.  Now suppose $r(X) > 0$.

By Lemma \ref{lem:exact-sequences}(\ref{eq:first-exact-sequence-map}) for $i = r(X)$, it suffices to prove that $h_{d,W}$ is an isomorphism, where $0 \le d \le \dim(X)$ and $W = (\Z/p)^{r(X)} \subseteq S$.  By Lemma \ref{lem:exact-sequences}(\ref{eq:second-exact-sequence-map}), it suffices to prove the lemma for $N_{d,W}$ and $N_{d,W} - X_{d,W}$.

By $\AA^1$-homotopy invariance, $CH^*_SN_{d,W} \cong CH^*_SX_{d,W} \cong CH^*_W(X_{d,W}/W')$, where $S = W \oplus W'$.  Since $W$ acts trivially on $X_{d,W}/W'$, $CH^*_W(X_{d,W}/W') \cong CH^*_W \otimes CH^*(X_{d,W}/W')$, and the lemma for $N_{d,W}$ follows from Lemmas \ref{lem:l_{S,*}}, \ref{lem:tv-for-bounded-above}, and \ref{lem:tensor-product-and-isomorphisms}.  

Since $X_{d,W}$ is the subscheme of $X$ with isotropy subgroup equal to $W$, $W'$ acts freely on $N_{d,W} - X_{d,W}$.  Moreover, the isotypic decomposition of $N_{d,W}$ contains no trivial representations of $W$.  Thus every point in $N_{d,W} - X_{d,W}$ has a proper subgroup of $W$ as its isotropy subgroup, so $r(N_{d,W} - X_{d,W}) < r(X)$, and the lemma holds for $N_{d,W} - X_{d,W}$ by induction on $r(X)$.
\end{proof}

\begin{lem}\label{lem:l_{G,XxGL(n)/S}}
Let $G \to GL(n)$ be a faithful representation, let $S \subset GL(n)$ be the subgroup of diagonal matrices of order 1 or $p$, and let $G$ act on $X \times GL(n)/S$ diagonally.  Then the map $\ell_{G,X\times GL(n)/S}$ is an isomorphism.
\end{lem}

\begin{proof}
For $\lambda \colon V \rightarrow G$ and $\rho \colon V \to S$, let ${}_\lambda Y$, $Y_\rho$, and ${}_\lambda Y_\rho$ be the functors from $k$-algebras to sets defined by 
\begin{align*}{}_\lambda Y(R) &= \{a \in GL_n(R) \mid a^{-1}(\im(\lambda)_R) a \subseteq S_R\}\\
Y_\rho(R) &= \{a \in GL_n(R) \mid a(\im(\rho)_R) a^{-1} \subseteq G_R\} \text{, and}\\
{}_\lambda Y_\rho(R) &= \{a \in GL_n(R) \mid a^{-1}\lambda a = \rho \colon V_R \to G_R \}.\end{align*}

Each of ${}_\lambda Y$, $Y_\rho$, and ${}_\lambda Y_\rho$ is represented by a closed subscheme of $GL(n)$, by \cite[Proposition 1.79]{Milne17}.  Indeed, in the notation of \cite[Section 1.i]{Milne17}, we have ${}_\lambda Y = T_{GL(n)}(\im(\lambda),S)$, $Y_\rho = T_{GL(n)}(\im(\rho),G)$, and ${}_\lambda Y_\rho = T_{GL(n)}(\lambda,\rho)$.  Here we view $\lambda$ and $\rho$ as closed points of $\underline{\Hom}_{\text{Sch}}(V,GL(n)) \cong GL(n)^{p^{\rk V}}$, and $GL(n)$ acts on $GL(n)$ and $\underline{\Hom}_{\text{Sch}}(V,GL(n))$ by conjugation.

Note that $(GL(n)/S)^\lambda = {}_\lambda Y/S$, $(G\backslash GL(n))^\rho = G\backslash Y_\rho$, ${}_\lambda Y = \coprod_{\rho \colon V \rightarrow S} {}_\lambda Y_\rho$, and $Y_\rho = \coprod_{\lambda \colon V \rightarrow G} {}_\lambda Y_\rho$.

Letting $H = GL(n)$, $C_\lambda = C_G(\lambda)$, and $R = \Rep(V,G)$, we have the following commutative diagram.
\begin{center}
\begin{tikzcd}
\displaystyle\coprod_{\lambda \in R}C_\lambda\backslash\left(BV \times EC_\lambda \times X^\lambda \times {}_\lambda Y \right)/S \arrow[r,"f_{G,X \times H/S}"] & G\backslash\left(EG \times X \times H\right)/S \\
\displaystyle\coprod_{\lambda \in R}C_\lambda\backslash (BV \times EC_\lambda \times ES \times X^{\lambda} \times {}_\lambda Y)/S \arrow[r,"\varphi"] \arrow[u,""] \arrow[d,""]& G\backslash (EG \times ES \times X \times H)/S \arrow[u,""] \arrow[d,""]\\
\displaystyle\coprod_{\lambda \in R} C_\lambda\backslash (BV \times ES \times X^\lambda \times {}_\lambda Y)/S \arrow[d,"\theta"] & G\backslash (ES \times X \times H)/S \\
\displaystyle\coprod_{\substack{\rho \colon V \rightarrow S\\ \lambda \colon V \rightarrow G}} G\backslash (BV \times ES \times X^\lambda \times {}_\lambda Y_{\rho})/S \arrow[ur,swap,"f_{S,(X \times H)/G}"] & 
\end{tikzcd}
\end{center}
The unlabeled vertical maps are all $\AA^1$-homotopy equivalences induced by projection.    On $BV \times EC_\lambda \times ES \times X^\lambda \times {}_\lambda Y_\rho$, $\varphi$ is induced by inclusion and the map $V \times C_\lambda \times S \rightarrow G \times S$, $(v,h,s) \mapsto (\lambda(v)h, \rho(v)s)$.  The map $\theta$ is an isomorphism induced by the identity maps $BV \times ES \times X^\lambda \times {}_\lambda Y_\rho \rightarrow BV \times ES \times X^\lambda \times {}_\lambda Y_\rho$.

Since $X$ has an equivariant ample line bundle, the quotient $(X \times H)/G$ exists as a scheme by geometric invariant theory.  It is smooth because $X$ is smooth, by \cite[\href{https://stacks.math.columbia.edu/tag/05B5}{Lemma 05B5}]{SP}. Thus $\ell_{S,(X \times H)/G}$ is an isomorphism by Lemma \ref{lem:l_{S,X}}, and it follows that $\ell_{G,X \times H/S}$ is an isomorphism. 
\end{proof}

\begin{lem}\label{lem:chow-ring-of-x-times-gln-mod-s}
Let $\rho$ be an $m$-dimensional representation of $G$.  Let $G$ act on $GL(m)/S$ by multiplication, where $S\subseteq GL(m)$ is the subgroup of diagonal matrices of order 1 or $p$.  Then we have
$$CH^*_G(X \times GL(m)/S) \cong CH^*_G(X) \otimes_{CH^*_{GL(m)}} CH^*_{S},$$
where $CH^*_G(X)$ is a $CH^*_{GL(m)}$-module via the composition 
$$CH^*_{GL(m)} \xrightarrow{(B \rho)^*} CH^*_G \rightarrow CH^*_G(X).$$
\end{lem}

\begin{proof}
This is a straightforward generalization of the proof of \cite[Lemma 6.6]{Totaro14}, which establishes the lemma in the case that $G$ is finite and $X$ is a point.  First we consider the Chow ring of $Y := G\backslash (EG \times X \times GL(m)/T)$, where $T$ is the subgroup of diagonal matrices in $GL(m)$.  We have that $Y$ is the total space of a $GL(m)/T$-bundle over $X//G := G\backslash(EG \times X)$ given by the pullback of the bundle $BT$ over $BGL(m)$ by the map $X//G \rightarrow BG \rightarrow BGL(m)$.  By $\AA^1$-homotopy invariance, such a $GL(m)/T$-bundle has the same Chow ring as a $GL(m)/B$-bundle, where $B$ is the subgroup of upper triangular matrices in $GL(m)$. This is the same as a flag bundle $\Fl(E)$, where $E$ is the vector bundle corresponding to the given map $X//G \rightarrow BGL(m)$.  The Chow ring of $\Fl(E)$ was computed for a vector bundle over a smooth scheme in, for example, the proof of \cite[Theorem 2.13]{Totaro14}.  We have 
$$CH^*(Y) \cong CH^*(\Fl(E)) \cong CH^*_G(X)[y_1,\ldots,y_n]/(e_i(y_1,\ldots,y_n) = c_i),$$
where $|y_i| = 1$, $e_i$ is the $i$th elementary symmetric function, and $c_i$ is the $i$th Chern class of $E$.

Next, $G\backslash (EG \times X \times GL(m)/S)$ is a $(G_m)^n$-bundle over $Y$ corresponding to the $p$th powers of the line bundles $L_1, \ldots, L_n$ induced by the map $Y \rightarrow BT$.  Thus
\begin{align*}
	CH^*_G(X \times GL(m)/S) &\cong CH^*_G(X)[y_1,\ldots,y_n]/(e_i(y_1,\ldots, y_n) - c_i, c_1(L_j^p)) \\
	&\cong CH^*_G(X)[y_1,\ldots,y_n]/(e_i(y_1,\ldots, y_n) - c_i), 
\end{align*}
where the second isomorphism follows from the fact that $c_1(L_j^p) = pc_1(L_j) = 0$.  Since $CH^*_{GL(m)} \cong \F_p[c_1,\ldots,c_n]$, the right hand side is isomorphic to $CH_G(X) \otimes_{CH^*_{GL(m)}}CH^*_S$. 
\end{proof}

We will need the following lemma in this section and again in Section \ref{sec:d0andd1}.

\begin{lem}\label{lem:quillen-exact-sequence}
With the set up of Lemma \ref{lem:chow-ring-of-x-times-gln-mod-s}, the coequalizer diagram of $G$-schemes
\[\begin{tikzcd}
X \times GL(m)/S \times GL(m)/S \arrow[r,shift left = .75ex,"p_1"] \arrow[r,shift right = .75ex, swap, "p_2"] & X \times GL(m)/S \arrow[r,"p"] & X
\end{tikzcd}\]
induces an exact sequence
\begin{equation*}
0 \rightarrow CH^*_G(X) \xrightarrow{p^*} CH^*_G(X \times GL(m)/S)  \xrightarrow{p_1^* - p_2^*} CH^*_G(X \times GL(m)/S \times GL(m)/S)
\end{equation*}
\end{lem}

\begin{proof}
By Lemma \ref{lem:chow-ring-of-x-times-gln-mod-s}, the exact sequence becomes
\begin{equation*}
0 \rightarrow CH^*_GX \xrightarrow{p^*} CH^*_GX \otimes_{CH^*_{H}} CH^*_S \xrightarrow{p_1^* - p_2^*} CH^*_GX \otimes_{CH^*_{H}} CH^*_S \otimes_{CH^*_{H}} CH^*_S
\end{equation*}
Since $CH^*_S$ is a faithfully-flat $CH^*_{H}$-module, this is exact by faithfully flat descent \cite[\href{https://stacks.math.columbia.edu/tag/03OA}{Lemma 03OA}]{SP}.
\end{proof}

\begin{proof}[Proof of Theorem \ref{thm:l_{G,X}}]
Let $H = GL(n)$, $K = GL(n)/S$, and $C_\lambda = C_G(\lambda)$.  We have a coequalizer diagram of $G$-schemes
\[\begin{tikzcd}
X \times K \times K \arrow[r,shift left = .75ex,"p_1"] \arrow[r,shift right = .75ex, swap, "p_2"] & X \times K \arrow[r,"p"] & X
\end{tikzcd}\]
This gives a commutative diagram
\[\begin{tikzcd}
0 \arrow[r,""] & F_1(X) \arrow[r,"p^*"] \arrow[d,"\ell_{G,X}"] & F_1(X \times K) \arrow[r,"p_1^* - p_2^*"] \arrow[d,"\ell_{G,X \times K}"] & F_1(X \times K \times K) \arrow[d,"\ell_{G,X \times K \times K}"]\\
0 \arrow[r,""] & F_2(X) \arrow[r,"p^*"] & F_2(X \times K) \arrow[r,"p_1^* - p_2^*"] & F_2(X \times K \times K)
\end{tikzcd}\]
We will show that the rows are exact.  The two rightmost vertical maps are isomorphisms by Lemma \ref{lem:l_{G,XxGL(n)/S}}, so this will finish the proof.  The top row is $T_V$ applied to the sequence
\begin{equation*}
0 \rightarrow CH^*_G(X) \xrightarrow{p^*} CH^*_G(X \times K)  \xrightarrow{p_1^* - p_2^*} CH^*_G(X \times K \times K),
\end{equation*}
which is exact by Lemma \ref{lem:quillen-exact-sequence}.  Thus the top row is exact by the exactness of $T_V$. 

The bottom row is the sequence
\begin{equation*}
0 \rightarrow \prod_{\lambda} CH^*_{C_\lambda}X^\lambda \xrightarrow{p^*} \prod_{\lambda} CH^*_{C_\lambda}(X \times K)^\lambda \xrightarrow{p_1^* - p_2^*} \prod_{\lambda} CH^*_{C_\lambda}(X \times K \times K)^\lambda.
\end{equation*}
Using the notation from the proof of Lemma \ref{lem:l_{G,XxGL(n)/S}}, $K^{\lambda} = \coprod_{\rho} {}_\lambda Y_\rho/S$, where $\rho$ ranges over all group homomorphisms $V \rightarrow S$.  Thus it suffices to show that for a given $\rho \colon V \rightarrow S$ and $\lambda \colon V \rightarrow G$,
\begin{equation*}
0 \rightarrow CH^*_{C_\lambda}X^\lambda \xrightarrow{p^*} CH^*_{C_\lambda}(X^\lambda \times {}_\lambda Y_\rho/S) \xrightarrow{p_1^* - p_2^*} CH^*_{C_\lambda}(X^\lambda \times {}_\lambda Y_\rho/S \times {}_\lambda Y_\rho/S)
\end{equation*}
is exact.  For a matrix $a \in {}_\lambda Y_\rho$, left multiplication by $a^{-1}$ defines an isomorphism ${}_\lambda Y_\rho \xrightarrow{\cong} C_{H}(\im(\rho)) \cong GL(n-k) \times G_m^k$, where $k = \rk(\im(\rho))$.  Since $k$ contains the $p$th roots of unity, $G_m/(\Z/p) \cong G_m$, and --- using that $G_m$ satisfies the Chow K{\"u}nneth property and that $CH^*G_m$ is trivial --- we have
$$CH^*_{C_{H}(\lambda)}(X^\lambda \times {}_\lambda Y_\rho/S) \cong CH^*_{C_{H}(\lambda)}(X^\lambda \times GL(n-k)/S'),$$
and
$$CH^*_{C_{H}(\lambda)}(X^\lambda \times {}_\lambda Y_\rho/S \times {}_\lambda Y_\rho/S) \cong CH^*_{C_{H}(\lambda)}(X^\lambda \times GL(n-k)/S' \times GL(n-k)/S'),$$
where $S = S' \times \im(\rho)$.  Thus the sequence is exact by Lemma \ref{lem:quillen-exact-sequence}.
\end{proof}

\section{Localization away from \texorpdfstring{$n$}{n}-nilpotent modules}\label{sec:localization}
In this section, we compute the localization of $CH^*_G$ away from $\Nil_n = \Nil_{2n}^{\top} \cap \U$ (Theorem \ref{thm:localization-away-from-nil-p}).  One consequence of this result is a generalization of the results of Yagita and Totaro on an algebro-geometric analog of Quillen's stratification theorem (Corollary \ref{cor:quillens-theorem}).

For $\I$ a set of injective objects in an abelian category $\A$, let $\N(\I)$ be the subcategory of $\A$ comprising the objects $A \in \A$ such that $\Hom_\A(A,I) = 0$ for all $I \in \I$.  We say that $A \in \A$ is \emph{$\N(\I)$-closed} if whenever $f \colon B \rightarrow C$ is a morphism with $\ker(f)$ and $\coker(f)$ in $\N(\I)$, the induced map 
$$\Hom_{\A}(C,A) \rightarrow \Hom_{\A}(B,A)$$
is an isomorphism.  Equivalently, $A$ is $\N(\I)$-closed if $\Ext^i_\A(N,A) = 0$ for all $N \in \N(\I)$ and $i = 0,1$.

Subcategories of the form $\N(\I)$ are always Serre subcategories.  We say that $\N(\I)$ is \emph{localizing} if the functor $r \colon \A \rightarrow \A/\N(\I)$ has a right adjoint $s$.  In this case, the localization functor $L = s \circ r$ and unit $\lambda \colon \id_{\A} \rightarrow L$ are characterized up to natural equivalence by the facts that for every $A \in \A$, the kernel and cokernel of $\lambda \colon A \rightarrow LA$ are in $\N(\I)$ and $LA$ is $\N(\I)$-closed.

Note that the inclusion functor $\U \rightarrow \U^{\top}$ has a right adjoint $\tilde{\mO}$:  for $M \in \U^{\top}$, $\tilde{\mO} M$ is the largest submodule of $M$ concentrated in even degrees.

\begin{lem}\label{lem:n-prime-nilpotents}
  Let $M \in \U$, and let $n\ge 0$ be an integer.  The following are equivalent.
   \begin{enumerate}[a)]
    \item $M \in \Nil_n$
    \item The operations $\Sq_a$ (for $p = 2$) or $P_a$ (for $p > 2$) are locally nilpotent for all $0 \le a < n$.
    \item $\Hom_{\U}(M,CH^*_V \otimes \tilde{\mO}J(2k)) = 0$ for all $V$ and for all $0\le k < n$
    \item $\Hom_{\U}(M,CH^*_V \otimes N) = 0$ for all $V$ and for all $N\in \U$ concentrated in degrees at most $n - 1$.
    \item $(T_VM)^i = 0$ for all $V$ and for all $i < n$.
   \end{enumerate}
\end{lem}

\begin{proof}
The implications $(a) \Leftrightarrow (b)$ and $(a) \Leftrightarrow (e)$ follow immediately from Proposition \ref{prop:n-nilpotents} and Lemma \ref{lem:compare-tv-and-tv-prime}.  The implications $(a) \Leftrightarrow (c)$ and $(a) \Leftrightarrow (d)$ follow from Proposition \ref{prop:n-nilpotents} and the fact that, by \cite[Proposition 8.3]{LZ86}, $\tilde{\mO}(H^*_V \otimes N) = CH^*_V \otimes \tilde{\mO}N$ for all $N \in \U^{\top}$.
\end{proof}

The unstable $\A$-modules $H^*_V \otimes J(2k) \in \U^{\top}$ are injective in $\U^{\top}$ for all $k$ and for all $V \in \E$, so the unstable $\A$-modules $CH^*_V \otimes \tilde{\mO} J(2k) = \tilde{\mO}(H^*_V \otimes J(2k))$ are injective in $\U$.  Let $\I_n$ be the set of objects $CH^*_V \otimes \tilde{\mO}J(2k)$ in $\U$ for $V \in \E$ and $0 \le k < n$.  By Lemma \ref{lem:n-prime-nilpotents}, $\N(\I_n) = \Nil_n$.

\begin{lem}\label{lem:nil-prime-are-localizing}
The subcategories $\Nil_n$ are localizing for all $n$.
\end{lem}

\begin{proof}
Since $\Nil_n = \N(\I_n)$, $\Nil_n$ is a Serre subcategory.  Moreover, $\U$ has injective hulls, and every object $M \in \U$ has a subobject maximal among all subobjects in $\Nil_n$, because the same things are true for $\U^{\top}$ and $\Nil_{2n}^{\top}$ --- the injective hulls and maximal subobjects in $\U$ are given by applying $\tilde{\mO}$ to the respective constructions in $\U^{\top}$.  Thus by \cite[Corollaire III.3.1]{Gabriel62}, $r_n \colon \U \rightarrow \U/\Nil_n$ has a right adjoint $s_n$.
\end{proof}

\begin{lem}\label{lem:basic-niln-closed}
\begin{enumerate}[a)]
\item A limit of $\Nil_n$-closed unstable modules is $\Nil_n$-closed.
\item If $0 \rightarrow M_1 \rightarrow M_2 \rightarrow M_3 \rightarrow 0$ is an exact sequence in $\U$ and $M_1$ and $M_3$ or $M_2$ and $M_3$ are $\Nil_n$-closed, the third module is as well.
\item For all $N \in \U$ concentrated in degrees at most $n-1$ and $V \in \E$, $CH^*_V \otimes N$ is $\Nil_n$-closed. 
\end{enumerate}
\end{lem}

\begin{proof}
Parts (a) and (b) are straightforward.  Part (c) follows from (b) and the fact that $CH^*_V \otimes \tilde{\mO}J(2k)$ is $\Nil_{k+1}$-closed for all $k$ and for all $V \in \E$.
\end{proof}

The groupoid of maps from a smooth scheme $Y$ into the quotient stack $[X/G]$ is the category of diagrams of $G$-schemes of the form
$$\begin{tikzcd}
E \arrow[r] \arrow[d] & X\\
Y &
\end{tikzcd}$$
where $G$ acts freely on $E$ and trivially on $Y$, and $E \to Y$ is a principal $G$-bundle.  Given such a diagram and a map of smooth schemes $f \colon Z \to Y$, we have the diagram
$$\begin{tikzcd}
E \times_Y Z \arrow[r] \arrow[d] & X\\
Z &
\end{tikzcd}$$
inducing a functor $f^* \colon [X/G](Y) \to [X/G](Z)$.  Let $A^*_GX$ denote the ring of characteristic classes for these diagrams valued in equivariant Chow rings.  That is, $A^*_GX$ is the ring of assignments $\alpha \colon [X/G](Y) \to CH^*_GY$ that are natural with respect to the pullback of diagrams described above.

The following lemma is a straightforward generalization of \cite[Theorem 1.3]{Totaro99}.
\begin{lem}\label{lem:characteristic-class-interpretation}
The natural map $\varphi \colon A^*_GX \to CH^*_GX$ that sends $\alpha \in A^*_GX$ to
$$\alpha\left(\begin{tikzcd}
EG \times X \arrow[r] \arrow[d] & X\\
EG \times_G X &
\end{tikzcd}\right)$$
is an isomorphism.
\end{lem}

\begin{proof}
Given a diagram of $G$-schemes
$$\begin{tikzcd}
Y \arrow[d] \arrow[r] & X\\
Y/G
\end{tikzcd}$$
we have induced maps
$$Y/G \leftarrow EG \times_G Y \rightarrow EG \times_G X.$$
The first map is an $\AA^1$-homotopy equivalence, so these maps induce a map $CH^*_GX \to CH^*Y/G$, giving a map $\psi \colon CH^*_GX \to A^*_GX$.  It is straightforward to check that $\psi$ defines an inverse to $\varphi$.
\end{proof}

\begin{lem}\label{lem:induces-identity}
Let $h \in G$.  Let $G$ act on $EG \times X$ by $g\cdot(e,x) = (ge,gx)$ and on $EG \times EG \times X$ by $g\cdot (e_1,e_2,x) = (ge_1,hgh^{-1}e_2,gx)$.  Let $\pi_1,\pi_2 \colon (EG \times EG \times X)/G \to EG \times_G X$ be the maps given by $\pi_1(\overline{(e_1,e_2,x)}) = \overline{(e_1,x)}$ and $\pi_2(\overline{(e_1,e_2,x)}) = \overline{(e_2,hx)}$.  Then $\psi := \pi_2^*\circ (\pi_1^*)^{-1} \colon CH^*_GX \to CH^*_GX$ is the identity.
\end{lem}

\begin{proof}
We have a commutative diagram of $G$-schemes
$$\begin{tikzcd}
X & \arrow[l] EG \times X \arrow[d] & \arrow[l,swap,"f_1"] EG \times EG \times X \arrow[d] \arrow[r,"f_2"] & EG \times X \arrow[d] \arrow[r] & X\\
&EG \times_G X & \arrow[l,swap,"\pi_1"] (EG \times EG \times X)/G \arrow[r,"\pi_2"] & EG \times_G X &
\end{tikzcd}$$
where $f_1$ and $f_2$ are given by $(e_1,e_2,x) \mapsto (e_1,x)$ and $(e_1,e_2,x) \mapsto (h^{-1}e_2,x)$.  Both squares are pullback squares.

By Lemma \ref{lem:characteristic-class-interpretation}, an element $x \in CH^*_GX$ is equal to $\alpha(EG \times X)$ for some characteristic class $\alpha \in A^*_GX$ (where we suppress the maps to $X$ and $EG \times_G X$ from the notation).  From the diagram above, we have that
$$\psi^*(\alpha(EG\times X)) = (\pi_1^*)^{-1}(\pi_2^*(\alpha(EG\times X))) = \alpha(EG\times X)$$
as desired.
\end{proof}

For $E$ an elementary abelian subgroup of $G$, we have a map $CH^*_GX \rightarrow CH^*_E \otimes (CH^*_{C_G(E)}X^E)^{<n}$ given by composing the map $CH^*_GX \rightarrow CH^*_E \otimes CH^*_{C_G(E)}X^E$ induced by $E \times C_G(E) \rightarrow G$, $(e,g) \mapsto eg$ with the quotient map $CH^*_E \otimes CH^*_{C_G(E)}X^E \rightarrow CH^*_E \otimes (CH^*_{C_G(E)}X^E)^{<n}$.  For elementary abelian subgroups $E_1$ and $E_2$ of $G$ and $h \in G$ such that $hE_1h^{-1} \subseteq E_2$, we have a conjugation map $c_h \colon E_1 \rightarrow E_2$, $e \mapsto heh^{-1}$.  The group homomorphism 
\begin{align*} E_1 \times E_1 \times C_G(E_2) &\rightarrow E_1 \times C_G(E_1)\\
(e_1,e_2,g) &\mapsto (e_1e_2,h^{-1}gh),\end{align*}
together with the map $X^{E_2}\to X^{E_1}$, $x\mapsto h^{-1}x$ induces a map
$$\phi_1 \colon EE_1 \times EE_1 \times EC_G(E_2) \times_{C_G(E_2)} X^{E_2} \to EE_1 \times EC_G(E_1)\times_{C_G(E_1)} X^{E_1},$$
and the group homomorphism
\begin{align*} E_1 \times E_1 \times C_G(E_2) &\rightarrow E_2 \times C_G(E_2)\\
(e_1,e_2,g) &\mapsto (he_1h^{-1},he_2h^{-1}g).\end{align*}
together with the identity map $X^{E_2} \to X^{E_2}$ induces a map
$$\phi_2 \colon EE_1 \times EE_1 \times EC_G(E_2) \times_{C_G(E_2)} X^{E_2} \to EE_2 \times EC_G(E_2)\times_{C_G(E_2)} X^{E_2}.$$
By Lemma \ref{lem:induces-identity}, the compositions of the maps induced by $\phi_1$ and $\phi_2$ with the map $CH^*_GX \rightarrow CH^*_E \otimes (CH^*_{C_G(E)}X^E)^{<n}$ give the same composite map
$$CH^*_GX \rightarrow CH^*_{E_i} \otimes (CH^*_{C_G(E_i)}X^{E_i})^{<n} \rightarrow CH^*_{E_1} \otimes (CH^*_{E_1} \otimes CH^*_{C_G(E_2)}X^{E_2})^{<n}.$$

\begin{thm}\label{thm:localization-away-from-nil-p}
Let $\lambda_n$ be the map from $CH^*_GX$ to
$$\begin{tikzcd}
\Eq \colon \displaystyle\prod_{E}CH^*_E \otimes CH^{<n}_{C_G(E)}X^E \arrow[r,shift left = .75ex,""] \arrow[r,shift right = .75ex, swap, ""]& \displaystyle\prod_{c_h \colon E_1 \rightarrow E_2} CH^*_{E_1} \otimes (CH^*_{E_1} \otimes CH^*_{C_G(E_2)}X^{E_2})^{<n}
\end{tikzcd}$$
induced by the maps above.  Then $\lambda_n$ is localization away from $\Nil_n$.
\end{thm}

Note that Theorem \ref{thm:localization-away-from-nil-p} doesn't follow from \cite[Theorem 4.9]{HLS93}, since $CH^*_GX$ need not be finitely generated as a module over $CH^*_G$, and $CH^*_G$ need not be Noetherian.

\begin{proof}
By Lemma \ref{lem:basic-niln-closed}, $CH^*_E \otimes N$ is $\Nil_n$-closed for all $E \in \E$ and $N$ concentrated in degrees less than $n$.  The product of $\Nil_n$-closed unstable modules is $\Nil_n$-closed, as is the kernel of a map between $\Nil_n$-closed unstable modules.  It follows that the codomain of $\lambda_n$ is $\Nil_n$-closed, so it remains to show that $\ker(\lambda_n)$ and $\coker(\lambda_n)$ are in $\Nil_n$.  By Lemma \ref{lem:n-prime-nilpotents}, it is equivalent to show that $T_V(\lambda_n)$ is an isomorphism in degrees less than $n$ for all $V \in \E$.

Applying lemmas \ref{lem:tv-for-bounded-above}, \ref{lem:l_{S,*}}, and \ref{lem:tensor-product-and-isomorphisms}, $T_V(\lambda_n)$ becomes
\begin{align*}&\prod_{[\rho] \in \Rep(V,G)} CH^*_{C_G(\rho)}X^\rho \xrightarrow{\varphi}\\
&\ker\left(\prod_{E,\rho'}CH^*_E \otimes CH^{<n}_{C_G(E)}X^E \xrightarrow{\varphi_1 - \varphi_2} \prod_{c_h,\rho''}CH^*_{E_1} \otimes (CH^*_{E_1} \otimes CH^*_{C_G(E_2)}X^{E_2})^{<n}\right),\end{align*}
where $\rho'$ ranges over all maps $V \rightarrow E$ and $\rho''$ ranges over all maps $V \rightarrow E_1$.  For two group homomorphisms $\rho$ and $\rho' \colon V \rightarrow E$, the map $CH^*_{C_G(\rho)}X^\rho \rightarrow CH^*_E \otimes CH^{<n}_{C_G(E)}X^E$ given by $\varphi$ is trivial if $\rho$ and $\rho'$ are not conjugate.  If $c_h(\rho') = \rho$ for some $h \in G$, then the map is the one induced by 
$$\phi \colon E \times C_G(E) \xrightarrow{\operatorname{mult}} C_G(E) \xrightarrow{c_h} C_G(\rho)$$
and the map $X^E \to X^\rho$ given by $x \mapsto hx$.

For $\rho' \colon V \rightarrow E$, $\rho'' \colon V \rightarrow E_1$, and $\alpha \colon E_1 \rightarrow E_2$, the map $CH^*_E \otimes CH^{<n}_{C_G(E)}X^E \rightarrow CH^*_{E_1} \otimes (CH^*_{E_1} \otimes CH^*_{C_G(E_2)}X^{E_2})^{<n}$ given by $\varphi_1$ is trivial unless $E = E_1$ and $\rho' = \rho''$, in which case it is the map induced by $\phi_1$.  The map given by $\varphi_2$ is trivial unless $E = E_2$ and $c_h(\rho'') = \rho'$, in which case it is the map induced by $\phi_2$.

It suffices to show that
\begin{align*}0 \rightarrow \prod_{[\rho] \in \Rep(V,G)} CH^*_{C_G(\rho)}X^\rho \xrightarrow{\varphi} \prod_{\substack{E\\\rho' \colon V \rightarrow E}}CH^*_E \otimes CH^*_{C_G(E)}X^E\\
\xrightarrow{\varphi_1 - \varphi_2} \prod_{\substack{c_h \colon E_1 \rightarrow E_2\\\rho'' \colon V \rightarrow E_1}}CH^*_{E_1} \otimes CH^*_{E_1} \otimes CH^*_{C_G(E_2)}X^{E_2}\end{align*}
is exact.  We have already seen that $\varphi_1 \circ \varphi = \varphi_2 \circ \varphi$, so it remains to show that $\varphi$ is injective and that $\im(\varphi) \supseteq \ker(\varphi_1 - \varphi_2)$.

Consider the composition
$$\prod_{[\rho] \in \Rep(V,G)} CH^*_{C_G(\rho)}X^\rho \xrightarrow{\varphi} \prod_{\substack{E\\\rho' \colon V \rightarrow E}}CH^*_E \otimes CH^*_{C_G(E)}X^E \rightarrow \prod_{[\rho] \in \Rep(V,G)} CH^*_{C_G(\rho)}X^\rho,$$
where the maps $CH^*_E \otimes CH^*_{C_G(E)}X^E \rightarrow CH^*_{C_G(\rho)}X^\rho$ are trivial unless $\im(\rho) = E$, in which case they are induced by the inclusion $C_G(\rho) \rightarrow \im(\rho) \times C_G(\rho)$ and the map $X^\rho \to X^E$, $x \mapsto h^{-1}x$.  The composition
$$C_G(\rho) \rightarrow \im(\rho) \times C_G(\rho) \xrightarrow{\phi} C_G(\rho)$$
is the identity, so the induced composition on Chow groups is the identity, so $\varphi$ is injective.

Finally, let $(x_{E,\rho'}) \in \prod CH^*_E \otimes CH^*_{C_G(E)}X^E$ with $|x_{E,\rho'}| = k$ for all $E$ and $\rho'$.  Then
$$x_{E,\rho'} = y_{E,\rho',k} \otimes 1 + y_{E,\rho',k-1} \otimes z_{E,\rho',1} + \ldots + 1 \otimes z_{E,\rho',k}$$
for some $y_{E,\rho',i}$ and $z_{E,\rho',j}$.  We will show that $\varphi((z_{\im(\rho),\rho,k})) = (x_{E,\rho'})$.  For a given $E$ and $\rho'$, consider the composition
$$CH^*_E \otimes CH^*_{C_G(E)}X^E \xrightarrow{\varphi_1 - \varphi_2} CH^*_E \otimes CH^*_E \otimes CH^*_{C_G(E)}X^E \rightarrow CH^*_E \otimes CH^*_{C_G(E)}X^E,$$
where the second map is induced by the identity on $X^E$ and the inclusion 
$$E \times C_G(E) \cong 1 \times E \times C_G(E) \rightarrow E \times E \times C_G(E).$$
Since $(x_{E,\rho'}) \in \ker(\varphi_1 - \varphi_2)$, the first map sends $x_{E,\rho'}$ to 0.  On the other hand, the composition sends $x_{E,\rho'}$ to $x_{E,\rho'} - \varphi'(z_{E,\rho',k})$, where $\varphi' \colon CH^*_{C_G(E)}X^E \rightarrow CH^*E \otimes CH^*_{C_G(E)}X^E$ is the map induced by multiplication.  If $\rho$ is conjugate to $\rho'$ via $c_h$, the map $\varphi \colon CH^*_{C_G(\rho)}X^\rho \rightarrow CH^*_E \otimes CH^*_{C_G(E)}X^E$ factors as $\varphi' \circ \psi_h$, where $\psi_h \colon CH^*_{C_G(\rho)}X^\rho \to CH^*_{C_G(E)}X^E$ is the map induced by $C_G(\rho) \to C_G(E)$, $g \mapsto hgh^{-1}$ and $X^\rho \to X^E$, $x \mapsto hx$.  Thus it suffices to show that $\psi_h^*(z_{\im(\rho),\rho,k}) = z_{E,\rho',k}$.  Consider the compositions
$$CH^*_{\im(\rho)} \otimes CH^*_{C_G(\rho)}X^\rho \xrightarrow{\varphi_1} CH^*_{\im(\rho)} \otimes CH^*_{\im(\rho)} \otimes CH^*_{C_G(E)}X^E \rightarrow CH^*_{C_G(E)}X^E,$$
and
$$CH^*_{E} \otimes CH^*_{C_G(E)}X^E \xrightarrow{\varphi_2} CH^*_{\im(\rho)} \otimes CH^*_{\im(\rho)} \otimes CH^*_{C_G(E)}X^E \rightarrow CH^*_{C_G(E)}X^E,$$
where the second maps are induced by the inclusion $C_G(E) \rightarrow \im(\rho) \times \im(\rho) \times C_G(E)$ and the identity on $X^E$.  The first composition sends $x_{\im(\rho), \rho}$ to $\psi_h^*(z_{\im(\rho),\rho,k})$, while the second sends $x_{E,\rho'}$ to $z_{E,\rho',k}$.  Since $(x_{E,\rho'}) \in \ker(\varphi_1 - \varphi_2)$, it follows that $\psi_h^*(z_{\im(\rho),\rho,k}) = z_{E,\rho',k}$.
\end{proof}

Taking $n = 1$ in Theorem \ref{thm:localization-away-from-nil-p} gives the following corollaries.

\begin{cor}\label{cor:quillens-theorem}
The map
$$CH^*_G \rightarrow \lim_{E\subseteq G} CH^*_E$$
is an F-isomorphism.  Here the limit ranges over all elementary abelian subgroups $E$ of $G$, with maps given by inclusion and conjugation.
\end{cor}

\begin{proof}
This follows from Theorem \ref{thm:localization-away-from-nil-p}, Lemma \ref{lem:n-prime-nilpotents}, and the fact that the domain and codomain are unstable $\A$-algebras.
\end{proof}

\begin{cor}
If $G$ is a linear algebraic group over $\C$, then the cycle map
$$CH^*_G \rightarrow H^*_G$$
is an F-isomorphism.
\end{cor}

\begin{proof}
This follows from Quillen's stratification theorem, Corollary \ref{cor:quillens-theorem}, and the fact that the cycle map $CH^*_E \rightarrow H^*_E$ is an F-isomorphism for every $E \in \E$.
\end{proof}
\section{The invariants \texorpdfstring{$d_0(CH^*_G)$}{d0(CH*G)} and \texorpdfstring{$d_1(CH^*_G)$}{d1(CH*G)}}\label{sec:d0andd1}
Let $L_n$ denote localization away from $\Nil_n$.  For $M \in \U$, define $d_0(M)$ and $d_1(M)$ to be the smallest integers $n_0$ and $n_1$ such that the localization map $\lambda_{n_i+1} \colon M \to L_{n_i + 1}M$ is injective and an isomorphism, respectively.  If no such $n_i$ exists, we define $d_i(M) = \infty$.  In this section, we prove Totaro's conjecture regarding $d_0(CH^*_G)$ (Proposition \ref{prop:totaro-conjecture}) and give bounds for $d_0(CH^*_G)$ and $d_1(CH^*_G)$ in terms of the representation theory of $G$ (Theorem \ref{thm:localization-bounds}).

For $G$ a finite group, Totaro defined $d_0(CH^*_G)$ to be the smallest integer $d$ such that the map
$$CH^*_G \rightarrow \prod_{E}CH^*_E \otimes CH^{\le d}_{C_G(E)}$$
is injective, and conjectured that $d_0(CH^*_G)$ is equal to the largest integer $d$ such that $CH^*_G$ contains a $d$th suspension of an unstable $\A$-module \cite[Conjecture 12.8]{Totaro14}.  By Theorem \ref{thm:localization-away-from-nil-p}, the two definitions of $d_0(CH^*_G)$ agree, and the following proposition positively answers Totaro's conjecture.

\begin{prop}\label{prop:totaro-conjecture}
The invariant $d_0(CH^*_G)$ is equal to the largest $d$ such that $CH^*_G$ contains a nonzero submodule of the form $\Sigma^dM$, for $M$ an unstable $\A$-module.
\end{prop}

\begin{proof}
Note that $d_0(CH^*_G)$ is equal to the largest $d$ such that $M_d \ne 0$, where $M_d \subseteq CH^*_G$ is the largest unstable submodule of $CH^*_G$ contained in $\Nil_d$.  (Such a $d$ exists by Theorem \ref{thm:localization-bounds}, which will give an upper bound for $d_0(CH^*_G)$.)  By \cite[Lemma 6.1.4]{Schwartz94}, $M_d/M_{d+1}$ is the $d$th suspension of an unstable $\A$-module.  On the other hand, by Proposition \ref{prop:n-nilpotents}, any unstable module of the form $\Sigma^dM$, $M \in \U$, is $2d$-nilpotent when regarded as an object of $\U^{\top}$, hence $d$-nilpotent as an object of $\U$.
\end{proof}

\begin{thm}\label{thm:localization-bounds}
Suppose that $G$ has a faithful representation of degree $n$.  Then 
\begin{enumerate}[a)]
\item $d_0(CH^*_GX) \le n^2 + \dim X - \dim G$ 
\item $d_1(CH^*_GX) \le 2n^2 + \dim X - \dim G$
\end{enumerate}
Suppose $G$ is finite.  Then 
\begin{enumerate}[a)]
\setcounter{enumi}{2}
\item $d_0(CH^*_GX) \le n(n-1)/2 + \dim X$
\item $d_1(CH^*_GX) \le n(n-1) + \dim X$
\end{enumerate}
\end{thm}

I expect that the bounds of Theorem \ref{thm:localization-bounds} can be improved.  By \cite[Theorem 12.7]{Totaro14}, $d_0(CH^*_G) \le n - c$, where $G$ is a finite $p$-group, $n$ is the degree of a faithful representation of $G$, and $c$ is the $p$-rank of the center of $G$.  Furthermore, Totaro and Guillot (\cite{Totaro14}, \cite{Guillot05}) have observed that $d_0(CH^*_G)$ is often less than known bounds in examples, particularly when $p$ is odd.  In another direction, Heard has given bounds for $d_0(M)$ for certain \emph{Noetherian} $M \in \K^{\top}$ (\cite{Heard20}).

\begin{lem}\label{lem:quillen-exact-sequence-di-bounds}(cf \cite[Proposition 3.6(b)]{HLS93}) Let $M_1, M_2, M_3 \in \U$.
\begin{enumerate}[a)]
\item If $0 \rightarrow M_1 \rightarrow M_2 \rightarrow M_3$ is an exact sequence, then $d_0M_1 \le d_0M_2$ and $d_1M_1 \le \max(d_1M_2, d_0M_3)$.
\item If $0 \rightarrow M_1 \rightarrow M_2 \rightarrow M_3 \rightarrow 0$ is an exact sequence, then $d_1M_2 \le \max(d_1M_1,d_1M_3)$.
\end{enumerate}
\end{lem}

\begin{proof}
The localization functor $L_n$ is given by the composition of an exact functor with its right adjoint, so it is left exact.  The bounds in (a) follow from the five lemma.  Lemma \ref{lem:basic-niln-closed} implies (b).
\end{proof}

\begin{lem}\cite[Lemma 5.3]{Totaro14}\label{lem:t-action-knockdown}
Suppose a split torus $T$ acts on a smooth scheme $X$ with finite stabilizer groups.  Then $CH^iX = 0$ for $i > \dim X - \dim T$.
\end{lem}

\begin{lem}\label{lem:d1-bound-elementary-abelian-case}
Suppose $S \in \E$ acts on a smooth scheme $X$, $T$ is a split torus acting on $X$ with finite stabilizer groups, and the actions of $S$ and $T$ commute.  Then $d_1(CH^*_SX) \le \dim X - \dim T$.  In particular, $d_1(CH^*_SX) \le \dim X$ for any smooth $S$-scheme $X$.
\end{lem}

Note that, since $d_0(M) \le d_1(M)$ for all $M \in \U$, Lemma \ref{lem:d1-bound-elementary-abelian-case} also gives a bound for $d_0(CH^*_SX)$.

\begin{proof}
Let $r(X) = \max_E(\rk E)$ where $E$ ranges over all isotropy subgroups of $S$ of points on $X$.  If $r(X) = 0$, then $S$ acts freely on $X$, so $CH^*_SX \cong CH^*X/S$.  As $T$ acts on $X/S$ with finite stabilizers, 
$$d_1(CH^*_SX) = d_1(CH^*X/S) \le \dim(X) - \dim(T)$$
by Lemmas \ref{lem:basic-niln-closed} and \ref{lem:t-action-knockdown}.  Now suppose that $r(X) > 0$.

By Lemma \ref{lem:exact-sequences} for $i = r(X)$ and Lemma \ref{lem:quillen-exact-sequence-di-bounds}, it suffices to prove the claim for $CH^*_SN_{d,W}$ and $CH^*_S(N_{d,W} - X_{d,W})$, where $0 \le d \le \dim(X)$ and $W = (\Z/p)^{r(X)} \subseteq S$.  (Since the action of $T$ commutes with the action of $S$, the action of $T$ restricts to an action on $X_{d,W}$, inducing actions on $N_{d,W}$ and $N_{d,W} - X_{d,W}$.)

By $\AA^1$-homotopy invariance, 
$$CH^*_SN_{d,W} \cong CH^*_SX_{d,W} \cong CH^*_W(X_{d,W}/W') \cong CH^*_W \otimes CH^*(X_{d,W}/W'),$$  
where $S = W \oplus W'$.  Thus the claim holds for $N_{d,W}$ by Lemmas \ref{lem:basic-niln-closed} and \ref{lem:t-action-knockdown}.

As in the proof of Lemma \ref{lem:l_{S,X}}, $r(N_{d,W} - X_{d,W}) < r(X)$, so the claim holds for $N_{d,W} - X_{d,W}$ by induction.
\end{proof}

\begin{proof}[Proof of Theorem \ref{thm:localization-bounds}]
Let $G \rightarrow GL(n)$ be a faithful representation of $G$.  As in Section \ref{sec:Computation_of_T_VCH^*_GX}, let $S \subset GL(n)$ be the subgroup of diagonal matrices of order 1 or $p$.  By Lemma \ref{lem:quillen-exact-sequence}, we have an exact sequence 
\begin{equation*}
0 \rightarrow CH^*_G(X) \xrightarrow{p^*} CH^*_G(X \times GL(n)/S)  \xrightarrow{p_1^* - p_2^*} CH^*_G(X \times GL(n)/S \times GL(n)/S).
\end{equation*}
Since 
$$CH^*_G(X \times GL(n)/S) \cong CH^*_S(X \times_G GL(n))$$
and 
$$CH^*_G(X \times GL(n)/S \times GL(n)/S) \cong CH^*_{S \times S}((X \times GL(n) \times GL(n))/G),$$
parts (a) and (b) follow from Lemmas \ref{lem:quillen-exact-sequence-di-bounds} and \ref{lem:d1-bound-elementary-abelian-case}.

Now suppose that $G$ is finite.  Let $U \subseteq GL(n)$ be the subgroup of upper triangular matrices with 1's on the diagonal.  Since $G$ is finite, the quotients $X \times_G U \backslash GL(n)$ and $(X \times U\backslash GL(n) \times U\backslash GL(n))/G$ exist as a smooth schemes.  As $U$ is an iterated extension of copies of the additive group $G_a$, we have that 
$$CH^*_S(X \times_G GL(n)) \cong CH^*_S(X \times_G U\backslash GL(n)),$$
and
$$CH^*_{S \times S}((X \times GL(n) \times GL(n))/G) \cong CH^*_{S \times S}((X \times U\backslash GL(n) \times U\backslash GL(n))/G)$$
by $\AA^1$-homotopy invariance.  The subgroup $T \subseteq GL(n)$ of diagonal matrices normalizes $U$, so $T$ acts on $(X \times U\backslash GL(n))/G$ on the left by multiplication.  Moreover, the action of $T$ commutes with the action of $S$, so parts (c) and (d) follow from Lemmas \ref{lem:quillen-exact-sequence-di-bounds} and \ref{lem:d1-bound-elementary-abelian-case}.
\end{proof}

\bibliography{lannes-t-functor-references}

\begin{thebibliography}{{Sta}19}

\bibitem[Bri18]{Brion18}
Michel Brion.
\newblock Linearization of algebraic group actions.
\newblock In {\em Handbook of Group Actions. {V}ol. {IV}}, volume~41 of {\em
  Adv. Lect. Math. (ALM)}, pages 291--340. Int. Press, Somerville, MA, 2018.

\bibitem[Bro03]{Brosnan03}
Patrick Brosnan.
\newblock {S}teenrod operations in {C}how theory.
\newblock {\em Transactions of the American Mathematical Society},
  355:1869--1903, 2003.

\bibitem[Duf83]{Duflot83}
Jeanne Duflot.
\newblock Smooth toral actions.
\newblock {\em Topology}, 22(3):253--265, 1983.

\bibitem[EG98]{EG98}
Dan Edidin and William Graham.
\newblock Equivariant intersection theory (with an appendix by {A}ngelo
  {V}istoli: The {C}how ring of {M}2).
\newblock {\em Inventiones Mathematicae}, 131(3):595--634, 1998.

\bibitem[Ful98]{Fulton98}
William Fulton.
\newblock {\em Intersection Theory}.
\newblock Springer, 1998.

\bibitem[Gab62]{Gabriel62}
Pierre Gabriel.
\newblock Des cat{\'e}gories ab{\'e}liennes.
\newblock {\em Bulletin de la Soci{\'e}t{\'e} Math{\'e}matique de France},
  90:323--448, 1962.

\bibitem[Gui05]{Guillot05}
Pierre Guillot.
\newblock Steenrod operations on the {C}how ring of a classifying space.
\newblock {\em Advances in Mathematics}, 196(2):276--309, 2005.

\bibitem[Hea20]{Heard20}
Drew Heard.
\newblock The topological nilpotence degree of a {N}oetherian unstable algebra.
\newblock {\em arXiv:2003.13267}, 2020.

\bibitem[Hen01]{Henn01}
Hans-Werner Henn.
\newblock Cohomology of groups and unstable modules over the {S}teenrod
  algebra.
\newblock In {\em Homotopy Theoretic Methods in Group Cohomology. Advanced
  Courses in Mathematics CRM Barcelona (Institut d'Estudis Catalans Centre de
  Recerca Matem{\`a}tica).} Birkh{\"a}user, Basel, 2001.

\bibitem[HK{\O}17]{HKO17}
Marc Hoyois, Shane Kelly, and Paul~Arne {\O}stv{\ae}r.
\newblock The motivic {S}teenrod algebra in positive characteristic.
\newblock {\em Journal of the European Mathematical Society},
  19(12):3813--3849, 2017.

\bibitem[HLS95]{HLS93}
Hans-Werner Henn, Jean Lannes, and Lionel Schwartz.
\newblock Localizations of unstable ${A}$-modules and equivariant mod $p$
  cohomology.
\newblock {\em Mathematische Annalen}, 301:23--68, 1995.

\bibitem[KM97]{KM97}
Se{\'a}n Keel and Shigefumi Mori.
\newblock Quotients by groupoids.
\newblock {\em Annals of Mathematics}, 145(1):193--213, 1997.

\bibitem[Kuh13]{Kuhn13}
Nicholas Kuhn.
\newblock Nilpotence in group cohomology.
\newblock {\em Proceedings of the Edinburgh Mathematical Society}, 56:151--175,
  2013.

\bibitem[Lan85]{Lannes85}
Jean Lannes.
\newblock Sur la cohomologie modulo $p$ des $p$-groupes ab{\'e}liens
  {\'e}l{\'e}mentaires.
\newblock In {\em Proc. Durham Symposium on Homotopy Theory}, pages 97--116.
  Cambridge University Press, 1985.

\bibitem[Lan92]{Lannes92}
Jean Lannes.
\newblock Sur les espaces fonctionnels dont la source est le classifiant d'un
  $p$-groupe ab{\'e}lien {\'e}l{\'e}mentaire.
\newblock {\em Publications Math{\'e}matiques de l'I.H.{\'E}.S.}, 75:135--244,
  1992.

\bibitem[LZ86]{LZ86}
Jean Lannes and Sa{\"\i}d Zarati.
\newblock Sur les {$\mathcal{U}$}-injectifs.
\newblock {\em Annales Scientifiques de l'{\'E}cole Normale Sup{\'e}rieure},
  19(2):303--333, 1986.

\bibitem[Mil17]{Milne17}
James~S Milne.
\newblock {\em Algebraic groups: the theory of group schemes of finite type
  over a field}, volume 170.
\newblock Cambridge University Press, 2017.

\bibitem[MV99]{MV99}
Fabien Morel and Vladimir Voevodsky.
\newblock $\mathbb{A}^1$-homotopy theory of schemes.
\newblock {\em Publications Math\'ematiques de l'IH\'ES}, 90:45--143, 1999.

\bibitem[Qui71]{Quillen71}
Daniel Quillen.
\newblock The spectrum of an equivariant cohomology ring: {I}.
\newblock {\em Annals of Mathematics}, 94:549--572, 1971.

\bibitem[Rio12]{Riou12}
Jo{\"e}l Riou.
\newblock Op{\'e}rations de {S}teenrod motiviques.
\newblock {\em arXiv:1207.3121}, 2012.

\bibitem[Sch94]{Schwartz94}
Lionel Schwartz.
\newblock {\em Unstable Modules over the {S}teenrod Algebra and {S}ullivan's
  Fixed Point Set Conjecture}.
\newblock The University of Chicago Press, 1994.

\bibitem[{Sta}19]{SP}
The {Stacks project authors}.
\newblock The stacks project.
\newblock \url{https://stacks.math.columbia.edu}, 2019.

\bibitem[Sum74]{Sumihiro74}
Hideyasu Sumihiro.
\newblock Equivariant completion.
\newblock {\em Journal of Mathematics of Kyoto University}, 14:1--28, 1974.

\bibitem[Tot99]{Totaro99}
Burt Totaro.
\newblock The {C}how ring of a classifying space.
\newblock In {\em Proceedings of Symposia in Pure Mathematics}, volume~67,
  pages 249--281. American Mathematical Society, 1999.

\bibitem[Tot14]{Totaro14}
Burt Totaro.
\newblock {\em Group Cohomology and Algebraic Cycles}.
\newblock Cambridge University Press, 2014.

\bibitem[Tot16]{Totaro16}
Burt Totaro.
\newblock The motive of a classifying space.
\newblock {\em Geometry and Topology}, 20(4):2079--2133, 2016.

\bibitem[Voe03]{Voevodsky03}
Vladimir Voevodsky.
\newblock Reduced power operations in motivic cohomology.
\newblock {\em Publications Math{\'e}matiques de l'IH{\'E}S}, 98:1--57, 2003.

\bibitem[Yag02]{Yagita02}
Nobuaki Yagita.
\newblock Chow rings of classifying spaces of extraspecial $p$-groups.
\newblock In {\em Recent Progress in Homotopy Theory}, pages 397--409. American
  Mathematical Society, 2002.

\end{thebibliography}
\bibliographystyle{alpha}

\address
\end{document}